\documentclass[10pt, reqno]{amsart}
\usepackage[english]{babel}
\usepackage[colorlinks,citecolor=green,linkcolor=red]{hyperref}
\usepackage{amsthm}
\usepackage{color}
\usepackage{mathrsfs}
\usepackage{amsmath}
\usepackage{mathtools}
\usepackage{amsfonts}
\usepackage{amssymb}
\usepackage{bm}
\usepackage{physics}
\usepackage{enumitem}
\usepackage{tikz-cd}
\usepackage{a4wide}
\usepackage{cleveref}
\usepackage{a4wide}
\usepackage{esint}
\usepackage{nicefrac}
\usepackage{yfonts}

\newtheorem{thm}{Theorem}
\newtheorem*{thm*}{Theorem}
\newtheorem{lem}[thm]{Lemma}
\newtheorem{prop}[thm]{Proposition}

\theoremstyle{definition}
\newtheorem{defn}[thm]{Definition}

\theoremstyle{remark}
\newtheorem{rem}[thm]{Remark}

\newcommand{\RCD}{{\mathrm {RCD}}}

\newcommand{\TestV}{{\mathrm {TestV}}}

\newcommand{\LIP}{{\mathrm {LIP}}}
\newcommand{\BV}{{\mathrm {BV}}}

\newcommand{\diff}{{\mathrm{d}}}
\newcommand{\DIFF}{{\mathrm{D}}}
\newcommand{\lip}{{\mathrm {lip}}}

\DeclareMathOperator*{\aplimsup}{ap\,\varlimsup}
\DeclareMathOperator*{\apliminf}{ap\,\varliminf}

\newcommand{\XX}{{\mathsf{X}}}
\newcommand{\YY}{{\mathsf{Y}}}
\newcommand{\ZZ}{{\mathsf{Z}}}
\newcommand{\dist}{{\mathsf{d}}}
\newcommand{\mass}{{\mathsf{m}}}

\newcommand{\capa}{{\mathrm {Cap}}}

\newcommand{\HH}{\mathcal{H}}
\newcommand{\RR}{\mathbb{R}}
\newcommand{\NN}{\mathbb{N}}

\newcommand{\FF}{\mathcal{F}}
\newcommand{\GG}{\mathcal{G}}

\newcommand{\nchi}{{\raise.3ex\hbox{\(\chi\)}}}

\newcommand{\defeq}{\mathrel{\mathop:}=}

\newcommand{\mres}{\mathbin{\vrule height 1.6ex depth 0pt width 0.13ex\vrule height 0.13ex depth 0pt width 1.3ex}}

\newcommand{\fr}{\penalty-20\null\hfill$\blacksquare$}     

\def\Xint#1{\mathchoice
	{\XXint\displaystyle\textstyle{#1}}%
	{\XXint\textstyle\scriptstyle{#1}}%
	{\XXint\scriptstyle\scriptscriptstyle{#1}}%
	{\XXint\scriptscriptstyle\scriptscriptstyle{#1}}%
	\!\int}
\def\XXint#1#2#3{{\setbox0=\hbox{$#1{#2#3}{\int}$}
		\vcenter{\hbox{$#2#3$}}\kern-.5\wd0}}

\def\dashint{\Xint-}
\let\fint\dashint

\let\epsilon\varepsilon

\title{Subgraphs of $\rm BV$ functions on \(\rm RCD\) spaces}

\author[Gioacchino Antonelli]{Gioacchino Antonelli}
\address[Gioacchino Antonelli]{Courant Institute Of Mathematical Sciences (NYU), 251 Mercer Street, 10012, New York, USA}
\email{ga2434@nyu.edu}

\author[Camillo Brena]{Camillo Brena}
\address[Camillo Brena]{Scuola Normale Superiore, Piazza dei Cavalieri, 7, 56126 Pisa, Italy.}
\email{camillo.brena@sns.it}

\author[Enrico Pasqualetto]{Enrico Pasqualetto}
\address[Enrico Pasqualetto]{Scuola Normale Superiore, Piazza dei Cavalieri, 7, 56126 Pisa, Italy.}
\email{enrico.pasqualetto@sns.it}

\begin{document}

\date{\today}

\begin{abstract}
In this work we extend classical results for subgraphs of functions of bounded variation in $\mathbb R^n\times\mathbb R$ to the setting of $\XX\times\mathbb R$, where $\XX$ is an $\RCD(K,N)$ metric measure space. 

In particular, we give the precise expression of the push-forward onto $\XX$ of the perimeter measure of the subgraph in $\XX\times\mathbb R$ of a $\BV$ function on $\XX$. Moreover, in properly chosen good coordinates, we write the precise expression of the normal to the boundary of the subgraph of a $\BV$ function $f$ with respect to the polar vector of $f$, and we prove change-of-variable formulas.
\end{abstract}

\maketitle
\tableofcontents

\section{Introduction}

This short note is about the study of functions of bounded variation in the setting of $\RCD$ spaces. The study of analytic and geometric properties of $\RCD$ metric measure spaces $(\XX,\dist,\mass)$ flourished in the last decade, see the account in \cite{AmbICM} and the references \cite{Sturm06I, Sturm06II, Lott-Villani09, Ambrosio_2014, Gigli14, AmbrosioGigliMondinoRajala12,Erbar-Kuwada-Sturm13, AmbrosioMondinoSavare13, CavMil16}.
The geometric structure of these spaces up to $\mass$-negligible sets is pretty well-understood after the works \cite{Mondino-Naber14,GP16-2,DPMR16,KelMon16,bru2018constancy}. Recently, the research on these spaces has been focusing also on the study of structure results for sets of (locally) finite perimeter, and of fine properties of functions of (locally) bounded variation, see \cite{ambrosio2018rigidity, bru2019rectifiability, bru2021constancy, BGBV}.\\
We stress that this theory has recently found interesting applications in the study of the isoperimetric problem on non-compact smooth Riemannian manifolds with Ricci curvature bounded from below, see \cite{ConcavitySharpAPPS}, and in the proof of the Rank-One Theorem in this low regularity setting \cite{ABPrank}. We refer the reader to \Cref{sec:RCDspaces} for more details.
\smallskip

We fix from now on an $\RCD(K,N)$ metric measure space $(\XX,\dist,\mass)$.
Here $K\in \mathbb R$ plays the role of (synthetic) lower bound on the Ricci curvature, $N\in [1,\infty)$ plays the role of synthetic upper bound on the dimension. 
Given a function of locally bounded variation $f\in\BV_{\mathrm{loc}}(\XX)$ (see \Cref{sub:BVCalculus}) we consider in $\XX\times\mathbb R$ the subgraph 
$$
\GG_f\defeq\{(x,t)\in\XX\times\RR\,:\,t<f(x)\}.
$$
Notice that, under the sole assumption of $f$ being measurable, $f\in\BV_{\mathrm{loc}}(\XX)$ if and only if $\mathcal{G}_f$ is of locally finite perimeter, see the first part of the main \Cref{thm:Principale1} below. In this note we are concerned with studying the relation of the measure $|\DIFF \nchi_{\mathcal{G}_f}|$ on $\XX\times\mathbb R$ with the measure $|\DIFF f|$ on $\XX$. We will denote with $\pi^1:\XX\times\mathbb R\to\XX$ the projection map onto $\XX$, and with $\pi^2:\XX\times\mathbb R\to\mathbb R$ the projection map onto $\mathbb R$.
\smallskip

In the Euclidean setting, this study can be dated back at least to \cite{Miranda1964}. There, the author was concerned with the study of {\em Cartesian surfaces}, i.e., subsets of $\mathbb R^n\times\mathbb R$ that can be written as $\{(x,t)\in\mathbb R^n\times\mathbb R:x\in\Omega,\,t=f(x)\}$, where $\Omega\subseteq\mathbb R^n$ is open, and $f\in\mathrm{BV}_{\mathrm{loc}}(\Omega)$. A systematic study of Cartesian surfaces and subgraphs of functions of locally bounded variation in Euclidean spaces can be found in \cite[Section 4.1.5]{GMSCartCurr}. In fact, our results are the generalization of the results contained in \cite[Section 4.1.5]{GMSCartCurr} to the setting of finite-dimensional $\RCD$ spaces.
\smallskip

The results of \cite[Section 4.1.5]{GMSCartCurr} have been used in the short proof of the Rank-One Theorem in the Euclidean setting of \cite{MV19}. Moreover, outside the Euclidean setting, they have also been recently generalized in the setting of arbitrary Carnot groups in \cite[Theorem 1.3, Theorem 4.2, and Theorem 4.3]{DMV19}. The latter generalization has been exploited to prove the Rank-One Theorem for a subclass of Carnot groups, see \cite[Theorem 1.1 and Theorem 1.2]{DMV19}.
\smallskip

We aim now at stating the main results of this note. We recall some terminology and notation. We refer the reader to \Cref{PreciseRep} and \Cref{def:Decomposition} for more details. Given $f\in\mathrm{BV}_{\mathrm{loc}}(\XX)$ we can define in a natural way (see \Cref{PreciseRep}) the {\em approximate lower} and {\em upper limits} $f^\wedge(x)$, $f^\vee(x)$ of $f(x)$ at $x\in\XX$, and the precise representative $\bar f(x):=(f^\wedge(x)+f^\vee(x))/2$. The set of points $x\in\XX$ where $f^\wedge(x)<f^\vee(x)$ is called the {\em jump set} $J_f$. In the setting of finite-dimensional $\RCD$ spaces it always holds $\mass(J_f)=0$.

    We can write \(|\DIFF f|\) as \(|\DIFF f|^a+|\DIFF f|^s\),
	where \(|\DIFF f|^a\ll\mass\) and \(|\DIFF f|^s\perp\mass\). We also have \(|\DIFF f|^s=|\DIFF f|^j+|\DIFF f|^c\), where the \emph{jump part} is given by \(|\DIFF f|^j\defeq|\DIFF f|\mres J_f\), while the
	\emph{Cantor part} is given by \(|\DIFF f|^c\defeq|\DIFF f|^s\mres (\XX\setminus J_f)\), so that we can write $|\DIFF f|^c=|\DIFF f|\mres C_f$ with $\mass(C_f)=0$. Finally, we call $g_f$ the Borel function such that $|\DIFF f|^a=g_f\mass$.

We stress that in the low regularity setting of finite-dimensional $\RCD$ spaces we cannot give a pointwise meaning to $\DIFF f$ for a function $f$ of locally bounded variation. This also prevents us from proving the verbatim analogues of the results in \cite[Section 4.1.5]{GMSCartCurr} in our setting. Nevertheless, with the Calculus developed in the $\RCD$ setting, one can give a meaning to the polar vector $\nu_f$, that in the classical setting is $\DIFF f/|\DIFF f|$ {(\cite{BGBV}, after \cite{bru2019rectifiability})}. This $\nu_f$ belongs to the capacitary module $L^0_{\capa}(T\XX)$ (see \cite{debin2019quasicontinuous}), and it is defined through a divergence theorem with sufficiently smooth test vector fields, see \Cref{thm:DivergenceFormula}. When $f=\nchi_E$ for a locally finite perimeter set $E$, we denote $\nu_{\nchi_E}=:\nu_E$.
\smallskip

In the non-smooth setting of finite-dimensional $\RCD$ spaces there are no canonical local coordinates. Anyway, in a lot of situations, one can use the so-called {\em splitting maps}, see \Cref{def:Splitting}. Roughly speaking, a splitting map is a vector-valued harmonic map whose Jacobian matrix is close to be the identity in an integral sense, and whose Hessian matrix is close to be null in an integral sense. Splitting maps have been used to detect geometric properties of spaces with Ricci lower bounds since the seminal works \cite{Cheeger-Colding96, Cheeger-Colding97I}. They have also been used in the recent \cite{ChNa15, bru2018constancy, bru2019rectifiability, BrueNaberSemola20, bru2021constancy}. We also used the splitting maps in the proof of the Rank-One Theorem in finite-dimensional $\RCD$ spaces \cite{ABPrank}. The next definition is borrowed and inspired from the studies in \cite{bru2021constancy, ABPrank} and gives a good notion of local chart. Notice that as a consequence of \Cref{prop:Reduction} every function $f$ of locally bounded variation has total variation that is supported on the countable union of domains of good splitting maps.

\begin{defn}[Good splitting map]\label{goodsplitting}
Let $(\XX,\dist,\mass)$ be an $\RCD(K,N)$ space of essential dimension $n$. Take $\eta\in(0,n^{-1})$. Fix $y\in\XX$ and $r_y>0$.
We say that an $n$-tuple of harmonic $C_{K,N}$-Lipschitz maps  
$u=(u^1,\dots,u^n):B_{2 r_y}(y)\rightarrow\RR^n$ is a good $\eta$-splitting map on $D\subseteq\ B_{r_y}(y)$ if for every $x\in D$ and $s\in (0,r_y)$, $u$ is an $\eta$-splitting map on $B_s(x)$.
We simply write good splitting map if the value of $\eta\in(0,n^{-1})$ is not important.
\end{defn}
For the notion of essential dimension, we refer the reader to \Cref{sec:RCDspaces}.
Given a good splitting map, following \cite[Definition 3.5]{ABPrank}, we give the following definition. We are essentially reading the normals $\nu_f$ and $\nu_{\mathcal{G}_f}$ in charts.
\begin{defn}\label{def:Normals}
Let $(\XX,\dist,\mass)$ be an $\RCD(K,N)$ space of essential dimension $n$ and let $u$ be a good splitting map on $D\subseteq B_{r_y}(y)$. Let $f\in\BV_{\rm loc}(\XX)$. Then we define
\begin{enumerate}
	\item the $|\DIFF f|$-measurable map \(\nu_f^u\) defined at $|\DIFF f|$-a.e.\ $x\in B_{2r_y}(y)$ as $$\nu_f^u(x)\defeq ((\nu_f\,\cdot\, \nabla u^1)(x),\dots,(\nu_f\,\cdot\, \nabla u^n)(x)),$$
	\item the $|\DIFF\nchi_{\GG_f}|$-measurable map \(\nu_{\GG_f}^u\) defined at $|\DIFF \nchi_{\mathcal{G}_f}|$-a.e.\ $p:=(x,t)\in B_{2r_y}(y)\times\mathbb R$ as $$\nu_{\GG_f}^u(p)\defeq ((\nu_{\GG_f}\,\cdot\, \nabla u^1)(p),\dots,(\nu_{\GG_f}\,\cdot\, \nabla u^n)(p),(\nu_{\GG_f}\,\cdot\, \nabla \pi^2)(p)).$$
\end{enumerate}
\end{defn}
\medskip

We are now ready to state the main theorems of this note. In the first result we explicitly compute $\pi^1_*|\DIFF\nchi_{\mathcal{G}_f}|$ in $\XX$ in terms of $|\DIFF f|$. \Cref{thm:Principale1} is the generalization in the setting of finite-dimensional $\RCD$ spaces of \cite[Theorem 1 in Section 4.1.5]{GMSCartCurr}.

\begin{thm}\label{thm:Principale1}
Let $K\in\mathbb R$ and $N<\infty$. Let $(\XX,\dist,\mass)$ be an $\RCD(K,N)$ space, and let $f\in L^0(\mass)$. Then the following are equivalent:
\begin{itemize}
    \item $f\in \BV_{\rm loc}(\XX)$,
    \item $\GG_f$ has locally finite perimeter.
\end{itemize}
If this is the case, then
\begin{align*}
    \pi_*^1|\DIFF\nchi_{\GG_f}|=\sqrt{g_f^2+1}\,\mass+|\DIFF f|\mres(C_f\cup J_f).
\end{align*}
\end{thm}

In the following \Cref{thmmain1} we explicitly compute the normal to the boundary of the subgraph $\nu_{\mathcal{G}_f}^u$ in coordinates, with respect to the polar vector $\nu_f^u$ in coordinates, and the density $g_f$ of $|\DIFF f|^a$ with respect to $\mass$. For the sake of reference, the analogous result in the non-Euclidean setting of Carnot groups is obtained in \cite[Theorem 4.3]{DMV19}.

\begin{thm}\label{thmmain1}
Let $K\in\mathbb R$ and $N<\infty$. Let $(\XX,\dist,\mass)$ be an $\RCD(K,N)$ space, and let $f\in \BV_{\rm loc}(\XX)$. Let $u$ be a good splitting map on $D\subseteq B_{r_y}(y)$, where $y\in\XX$ and $r_y>0$. 

Then, for $|\DIFF\nchi_{\GG_f}|$-a.e.\ $(x,t)\in D\times\RR$, it holds that
$$
\nu_{\GG_f}^{{ u}}(x,t)=
\begin{cases}
	\Big(\sqrt{\frac{1}{1+g_f^2}}g_f\nu_f^{ u},-\sqrt{\frac{1}{1+g_f^2}}\Big)(x)\quad &\text{if }x\in D\setminus(J_f\cup C_f),
	\\
	(\nu_f^{ u},0)(x)\quad &\text{if }x\in D\cap(J_f\cup C_f).
\end{cases}
$$
\end{thm}
In the next \Cref{thmmain2} we extend \cite[Theorem 2 and Theorem 3 in Section 4.1.5]{GMSCartCurr} in the setting of finite-dimensional $\RCD$ spaces.
\begin{thm}\label{thmmain2}
Let $K\in\mathbb R$ and $N<\infty$. Let $(\XX,\dist,\mass)$ be an $\RCD(K,N)$ space of essential dimension $n$, and let $f\in \BV_{\rm loc}(\XX)$. Let $u$ be a good splitting map on $D\subseteq B_{r_y}(y)$, where $y\in\XX$ and $r_y>0$.
Let also $\varphi:D\times\RR\rightarrow \RR$ be a bounded Borel function. Then 
\begin{enumerate}[label=\roman*)]
	\item for every $i=1,\dots,n$,
	
\begin{align*}
	&\int_{(D\setminus J_f)\times\RR}\varphi(x,t) \big(\nu_{\GG_f}^{u}(x,t)\big)_{i} \dd{|\DIFF\nchi_{\GG_f}|(x,t)}\\&\qquad=\int_{D\setminus J_f}\varphi(x,\bar{f} (x))\big(\nu_{f}^u(x)\big)_{i}\dd{|\DIFF f|(x)},
\end{align*}

\item it holds

\begin{align*}
	&\int_{(D\setminus J_f)\times\RR}\varphi(x,t) \big(\nu_{\GG_f}^{u}(x,t)\big)_{n+1} \dd{|\DIFF\nchi_{\GG_f}|(x,t)}\\&\qquad=-\int_{D \setminus J_f}\varphi(x,\bar{f} (x))\dd{\mass(x)},
\end{align*}

\item for every $i=1,\dots,n$,
\begin{align*}
	&\int_{(D\cap J_f)\times\RR}\varphi(x,t) \big(\nu_{\GG_f}^u(x,t)\big)_{i} \dd{|\DIFF\nchi_{\GG_f}|(x,t)}\\&\qquad=\int_{D\cap J_f}\int_{f^\wedge(x)}^{f^\vee(x)}\varphi(x,t)\,\dd t\,\big(\nu_{f}^u(x)\big)_{i}\Theta_n(\mass,x)\dd{\HH^{n-1}(x)},
\end{align*}
where we set \(\Theta_n(\mass,x)\coloneqq\lim_{r\to 0}\mass(B_r(x))/r^n\),
\item it holds
$$\int_{(D\cap J_f)\times\RR}\varphi(x,t) \big(\nu_{\GG_f}^{u}(x,t)\big)_{n+1} \dd{|\DIFF\nchi_{\GG_f}|(x,t)}=0.$$
\end{enumerate}
\end{thm}
In \Cref{thmmain1} and \Cref{thmmain2}, we compare $\nu_{\GG_f}^u$ and $\nu_f^u$ only for a single good splitting map $u$, on its domain $D$. This, however, still allows us to have a complete picture (i.e.\ the comparison for $|\DIFF\nchi_{\GG_f}|$-a.e.\ $(x,t)$), thanks to the following result, taken from \cite[Lemma 2.27]{ABPrank}, which, in turn, is inspired by the techniques introduced in \cite{bru2019rectifiability}. The last part of the forthcoming statement is not explicitly written in \cite{ABPrank}, but it is a direct consequence of \eqref{diffandper}.
\begin{prop}\label{prop:Reduction}
Let $K\in\mathbb R$ and $N<\infty$. Let $(\XX,\dist,\mass)$ be an $\RCD(K,N)$ space of essential dimension $n$. Let also $\eta\in (0,n^{-1})$. Then there exists a family $\boldsymbol u_\eta=\{u_{\eta,k}\}_{k\in\mathbb N}$, where, for every $k\in\mathbb N$, $u_{\eta,k}$ is a good $\eta$-splitting map on $D_k\subseteq B_{r_k}(x_k)$, for some $x_k\in\XX$ and $r_k>0$, and moreover
    \begin{equation}\notag
	|\DIFF f|\bigg(\XX\setminus\bigcup_k D_k\bigg)=0,\qquad\text{for every }f\in\BV_{\rm loc}(\XX).
	\end{equation}
	In particular, it holds that
	$$
	|\DIFF\nchi_{\GG_f}|\bigg(\Big(\XX\setminus\bigcup_k D_k\Big)\times\RR\bigg)=0,\qquad\text{for every }f\in\BV_{\rm loc}(\XX).
	$$
\end{prop}

We spend a few lines about the strategy of the proof of \Cref{thmmain1}, as once it is obtained, 
\Cref{thmmain2} follows quite easily.  The classical strategy of \cite{GMSCartCurr} seems not suitable for our context, as we do not have a canonical way to decompose the distributional derivatives $\DIFF f$ and $\DIFF \nchi_{\GG_f}$ along different directions. This also causes the need to define the `components' $\big(\nu_{f}^u\big)_{i}$ $\big(\nu_{\GG_f}^u\big)_{i}$ exploiting maps that look like charts. The drawback is that these charts are defined only on Borel subsets, hence it is not clear the distributional nature of the objects $\big(\nu_{f}^u\big)_{i}$ and  $\big(\nu_{\GG_f}^u\big)_{i}$. Nevertheless, in our main result we compare $\big(\nu_{f}^u\big)_{i}$ with  $\big(\nu_{\GG_f}^u\big)_{i}$. In order to do so, a new strategy has to be exploited, and we therefore employ a blow-up procedure, which is more compatible with the use of Geometric Measure Theory results and does not need the distributional meaning of such objects. This strategy is in \Cref{sec:Main}, after \Cref{sec:Prel} in which some preliminary facts are discussed.

\subsection*{Acknowledgments}
The third named author acknowledges the support by the Balzan project led by Luigi Ambrosio.

\section{Preliminaries}\label{sec:Prel}
Given \(n\in\mathbb N\) and non-empty sets \(\XX_1,\dots,\XX_n\), for any \(i=1,\ldots,n\) we will denote by \(\pi^i\)
the projection of the Cartesian product \(\XX_1\times\dots\times\XX_n\) onto its \(i^{\rm th}\) factor:
\[
\pi^i\colon\XX_1\times\dots\times\XX_n\to\XX_i,\qquad(x_1,\dots,x_n)\mapsto x_i.
\]
\subsection{Metric measure spaces}
We say that a metric measure space \((\XX,\dist,\mass)\) is \emph{uniformly locally doubling} if for every radius \(R>0\)
there exists a constant \(C_D>0\) such that
\[
\mass(B_{2r}(x))\leq C_D\mass(B_r(x)),\qquad\text{ for every }x\in\XX\text{ and }r\in(0,R).
\]
Moreover, we say that \((\XX,\dist,\mass)\) supports a \emph{weak local \((1,1)\)-Poincar\'{e} inequality} if there exists
a constant \(\lambda\ge 1\) for which the following property holds: given any \(R>0\), there exists a constant \(C_P>0\) such that
for any function \(f\in{\rm LIP}_{\mathrm{loc}}(\XX)\) it holds that
\[
\fint_{B_r(x)}\bigg|f-\fint_{B_r(x)}f\,\diff\mass\bigg|\,\diff\mass\leq C_P r\fint_{B_{\lambda r}(x)}\lip f\,\diff\mass,
\qquad\text{ for every }x\in\XX\text{ and }r\in(0,R),
\]
where, for every $x\in \XX$,
\[
\lip f(x)\coloneqq\varlimsup_{y\to x}\frac{|f(x)-f(y)|}{\dist(x,y)},
\]
which has to be understood as $0$ if $x$ is isolated.

Uniformly locally doubling spaces supporting a weak local $(1,1)$-Poincar\'{e} inequality are usually called \emph{$\rm PI$ spaces}.

\subsubsection{BV calculus}\label{sub:BVCalculus}

We recall the notions of function of bounded variation and of finite perimeter set in the metric measure setting
following \cite{MIRANDA2003}.
\begin{defn}[Function of bounded variation]
Let $(\XX,\dist,\mass)$ be a metric measure space. Let $f\in L^1_{\mathrm{loc}}(\XX,\mass)$ be given. Then we define
\[
|\DIFF f|(\Omega) \coloneqq \inf\bigg\{\varliminf_{i\to\infty}\int_\Omega\lip f_i\,\diff\mass\;\bigg|\;\text{$(f_i)_{i\in\mathbb N}\subseteq{\rm LIP}_{\rm loc}(\Omega),
\,f_i \to f $ in $L^1_{\mathrm{loc}}(\Omega,\mass)$}\bigg\},
\]
for any open set $\Omega\subseteq\XX$. We declare that a function \(f\in L^1_{\rm loc}(\XX,\mass)\) is of \emph{local bounded variation},
briefly \(f\in{\rm BV}_{\rm loc}(\XX)\), if \(|\DIFF f|(\Omega)<+\infty\) for every \(\Omega\subseteq\XX\) open bounded. In this case, it is well-known that $|\DIFF f|$ extends to a locally finite measure on $\XX$. Moreover,
a function $f \in L^1(\mass)$ is said to belong to the space of \emph{functions of bounded variation}
${\rm BV}(\XX)={\rm BV}(\XX,\dist,\mass)$ if $|\DIFF f|(\XX)<+\infty$. 
\end{defn}
\begin{defn}[Set of finite perimeter]
Let \((\XX,\dist,\mass)\) be a metric measure space. Let \(E\subseteq\XX\) be a Borel set and \(\Omega\subseteq\XX\) an open set.
Then we define the \emph{perimeter} of $E$ in $\Omega$ as
\[
P(E,\Omega) \coloneqq \inf\bigg\{\varliminf_{i\to\infty}\int_\Omega \lip f_i\,\diff\mass\;\bigg|\;\text{$(f_i)_{i\in\mathbb N}\subseteq{\rm LIP}_{\rm loc}(\Omega),
\,u_i \to \nchi_E $ in $L^1_{\rm loc}(\Omega,\mass)$} \bigg\},
\]
in other words \(P(E,\Omega)\coloneqq|\DIFF\nchi_E|(\Omega)\).
We say that $E$ has \emph{locally finite perimeter} if $P(E,\Omega)<+\infty$ for every \(\Omega\subseteq\XX\) open bounded.
Moreover, we say that $E$ has \emph{finite perimeter} if $P(E,\XX)<+\infty$.
\end{defn}
The following coarea formula is taken from \cite{MIRANDA2003}.
\begin{thm}[Coarea]\label{coarea}
Let \((\XX,\dist,\mass)\) be a metric measure space. Let \(f\in L^1_{\rm loc}(\mass)\) be given. Then for any Borel set \(E\subseteq\XX\)
it holds that the function \(\RR\ni t\mapsto P(\{f>t\},E)\in[0,+\infty]\) is Borel measurable and
\[
|\DIFF f|(E)=\int_\RR P(\{f>t\},E)\,\dd t.
\]
In particular, it holds that \(f\in\BV(\XX)\) if and only if \(\{f>t\}\) is a set of finite perimeter for a.e.\ \(t\in\RR\)
and the function \(t\mapsto P(\{f>t\},\XX)\) belongs to \(L^1(\RR)\).
\end{thm}
We remark that the weak local $(1,1)$-Poincaré inequality of $\rm PI$ spaces holds for $\BV_{\rm loc}(\XX)$ functions $f$ in the following form: for every bounded open set $\Omega\subseteq \XX$ and every $R>0$, 
\[
\fint_{B_r(x)}\bigg|f-\fint_{B_r(x)}f\,\diff\mass\bigg|\,\diff\mass\leq C_P(N,K,\Omega,R) r\frac{|\DIFF f|({B_{\lambda r}(x)})}{\mass(B_{\lambda r}(x))},
\quad\text{ for every }x\in\Omega\text{ and }r\in(0,R).
\]
In the particular case of sets of (locally) finite perimeter, the $(1,1)$-Poincaré inequality reads as a local isoperimetric inequality, as can be shown with classical computations:
\begin{equation}\label{isop}
    \min\big\{\mass( B_r(x)\cap E),\mass( B_r(x)\setminus E)\big\}\le 2 C_P r|\DIFF\nchi_E|(B_{\lambda r}(x)),\quad\text{ for every }x\in\Omega\text{ and }r\in(0,R).
\end{equation}

The following proposition summarizes results about sets of finite perimeter that are now well-known in the context of PI spaces and are proved in \cite{amb01,erikssonbique2018asymptotic}, see also \cite{AmbrosioAhlfors}.
\begin{prop}\label{zuppa}
Let $(\XX,\dist,\mass)$ be a $\rm PI$ space and let $E\subseteq\XX$ be a set of locally finite perimeter. Then, for $|\DIFF\nchi_E|$-a.e.\ $x\in\XX$ the following hold:
\begin{enumerate}[label=\roman*)]
    \item $E$ is \emph{asymptotically minimal} at $x$, i.e., there exist $r_x>0$ and a function $\omega_x:(0,r_x)\rightarrow(0,\infty)$ with $\lim_{r\searrow 0}\omega_x(r)=0$ satisfying
    $$
    |\DIFF\nchi_E|(B_r(x))\le(1+\omega_x(r))|\DIFF\nchi_{E'}|(B_r(x)),\qquad\text{if $r\in (0,r_x)$ and $E'\Delta E\Subset B_r(x),$}
    $$
    \item $|\DIFF\nchi_E|$ is \emph{asymptotically doubling} at $x$, i.e.,
    $$\limsup_{r\searrow 0}\frac{|\DIFF\nchi_E| (B_{2 r}(x))}{|\DIFF\nchi_E| (B_{r}(x))}<\infty, $$
    \item we have the following estimates:
    $$
    0<\liminf_{r\searrow 0} \frac{r |\DIFF\nchi_E|(B_r(x))}{\mass(B_r(x))}\le \limsup_{r\searrow 0}\frac{r |\DIFF\nchi_E|(B_r(x))}{\mass(B_r(x))}<\infty,
    $$
    \item the following holds:
    $$ 
    \liminf_{r\searrow 0} \min\bigg\{\frac{\mass( B_r(x)\cap E)}{\mass(B_r(x))},\frac{\mass( B_r(x)\setminus E)}{\mass(B_r(x))}\bigg\}>0.$$
\end{enumerate}
\end{prop}
\medskip
We recall the following classical definition.
\begin{defn}[Precise representative]\label{PreciseRep}
Let $(\XX,\dist,\mass)$ be a metric measure space, and let $f\colon\XX\to\RR$ be a Borel function.		
Then we set the \emph{approximate lower} and \emph{upper limits} to be
\begin{alignat*}{3}
	f^{\wedge}(x)&\defeq \apliminf_{y\rightarrow x} f(y)&&\defeq\sup&&\left\{t\in\bar{\RR}\;:\;\lim_{r\searrow 0} \frac{\mass(B_r(x)\cap\{f<t\})}{\mass(B_r(x))}=0\right\}, \\
	f^{\vee}(x)&\defeq \aplimsup_{y\rightarrow x} f(y)&&\defeq\inf&&\left\{t\in\bar{\RR}\;:\;\lim_{r\searrow 0} \frac{\mass(B_r(x)\cap\{f>t\})}{\mass(B_r(x))}=0\right\},
\end{alignat*}
for every \(x\in\XX\). Here we are assuming by convention that
$$\inf\varnothing=+\infty\qquad \text{and}\qquad\sup\varnothing=-\infty.$$
Moreover, we define the \emph{precise representative} \(\bar f\colon\XX\to\bar\RR\) of \(f\) as
\[
\bar f(x)\coloneqq\frac{f^\wedge(x)+f^\vee(x)}{2},\qquad\text{ for every }x\in\XX,
\]
where we declare that \(+\infty-\infty=0\).
\end{defn}
We define the \emph{jump set} \(J_f\subseteq\XX\) of the function \(f\) as the Borel set
\[
J_f\coloneqq\big\{x\in\XX\,:\,f^\wedge(x)<f^\vee(x)\big\}.
\]
It is known that if \((\XX,\dist,\mass)\) is a PI space and \(f\in\BV(\XX)\), then \(\mass(J_f)=0\),
see \cite[Proposition 5.2]{ambmirpal04}.
Moreover, as proved in \cite[Lemma 3.2]{kinkorshatuo}, it holds that
\begin{equation*}
|\DIFF f|(\XX\setminus\XX_f)=0,\qquad\text{ where }\XX_f\coloneqq\big\{x\in\XX\;\big|\;-\infty<f^\wedge(x)\leq f^\vee(x)<+\infty\big\},
\end{equation*}
thus in particular \(-\infty<\bar f(x)<+\infty\) holds for \(|\DIFF f|\)-a.e.\ \(x\in\XX\).
\begin{defn}[Decomposition of the total variation measure]\label{def:Decomposition}
	Let \((\XX,\dist,\mass)\) be a PI space and let $f\in\BV_{\rm loc}(\XX)$.
	We write \(|\DIFF f|\) as \(|\DIFF f|^a+|\DIFF f|^s\),
	where \(|\DIFF f|^a\ll\mass\) and \(|\DIFF f|^s\perp\mass\). We can decompose the singular part \(|\DIFF f|^s\)
	as \(|\DIFF f|^j+|\DIFF f|^c\), where the \emph{jump part} is given by \(|\DIFF f|^j\defeq|\DIFF f|\mres J_f\), while the
	\emph{Cantor part} is given by \(|\DIFF f|^c\defeq|\DIFF f|^s\mres (\XX\setminus J_f)\), so that we can write $|\DIFF f|^c=|\DIFF f|\mres C_f$ with $\mass(C_f)=0$. Finally, we write $|\DIFF f|^a=g_f\mass$.
\end{defn}

We recall the definition of subgraph, and \cite[Lemma 2.11]{ABPrank}.
\begin{defn}
Let $(\XX,\dist,\mass)$ be a metric measure space and let $f:\XX\rightarrow\RR$ be Borel.
Then we define the \emph{subgraph} of $f$ as the Borel set \(\GG_f\subseteq\XX\times\RR\) given by
$$\GG_f\defeq\big\{(x,t)\in\XX\times\RR\,:\,t<f(x)\big\}.$$
\end{defn}

Before stating the next result, we remind that the \emph{essential boundary} \(\partial^*E\) of \(E\subseteq\XX\) Borel is
\[
\partial^*E\coloneqq\bigg\{x\in\XX\;\bigg|\;\limsup_{r\searrow 0}\frac{\mass(B_r(x)\cap E)}{\mass(B_r(x))}>0,\,
\limsup_{r\searrow 0}\frac{\mass(B_r(x)\setminus E)}{\mass(B_r(x))}>0\bigg\}.
\]
\begin{lem}\label{lem:char_G_f}
	Let $(\XX,\dist,\mass)$ be a uniformly locally doubling metric measure space and $f\in \BV_{\rm loc}(\XX)$. Then it holds that
	\begin{align*}
		({x},{t})\in\partial^*\GG_f&\quad\Rightarrow\quad t\in [f^\wedge(x),f^\vee(x)],\\
		t\in (f^\wedge(x),f^\vee(x))&\quad\Rightarrow\quad(x,t)\in\partial^*\GG_f.
	\end{align*}
	In particular, if \(x\in\XX_f\setminus J_f\), then it holds that \(\partial^*\GG_f\cap(\{x\}\times\RR)=\{(x,\bar f(x))\}\).
\end{lem}
\subsection{RCD spaces}\label{sec:RCDspaces}
We assume the reader is familiar with the theory of \(\RCD(K,N)\) spaces. Recall that an \(\RCD(K,N)\) space is an infinitesimally Hilbertian
metric measure space verifying the Curvature-Dimension condition \({\rm CD}(K,N)\), in the sense of Lott--Villani--Sturm, for some \(K\in\RR\)
and \(N\in[1,\infty]\). In this paper we only consider finite-dimensional \(\RCD(K,N)\) spaces, namely we assume \(N<\infty\).
Finite-dimensional \(\RCD\) spaces are PI. If not otherwise stated, through this note we will work in the setting of finite-dimensional $\RCD$ spaces. 

\subsubsection{Pointed measured Gromov--Hausdorff convergence and tangents}
Let us recall some classical facts about pointed measured Gromov--Hausdorff convergence and tangents in the setting of $\RCD(K,N)$ spaces. The exposition here is equivalent to the classical one in more general settings (see e.g.\ \cite{GMS15}), due to the result in \cite[Theorem 4.1]{AHT17}. See also the introduction given in \cite[Section 2.1.2]{ABPrank}.
\begin{defn}[Pointed measured Gromov-Hausdorff convergence]
Let \((\XX,\dist,\mass,p),(\XX_i,\dist_i,\mass_i,p_i)\),
for \(i\in\mathbb N\), be $\RCD(K,N)$ spaces with $K\in\mathbb R$ and $N<\infty$. Then we say that \((\XX_i,\dist_i,\mass_i,p_i)\to(\XX,\dist,\mass,p)\) in the \emph{pointed measured Gromov-Hausdorff sense}
(briefly, in the \emph{pmGH sense}) provided there exist a proper metric space \((\ZZ,\dist_\ZZ)\) and isometric embeddings
\(\iota\colon\XX\to\ZZ\) and \(\iota_i\colon\XX_i\to\ZZ\) for \(i\in\mathbb N\) such that \(\iota_i(p_i)\to\iota(p)\) and
\((\iota_i)_*\mass_i\rightharpoonup\iota_*\mass\) in duality with \(C_{\rm bs}(\ZZ)\), meaning that \(\int f\circ\iota_i\,\diff\mass_i
\to\int f\circ\iota\,\diff\mass\) for every \(f\in C_{\rm bs}(\ZZ)\). The space \(\ZZ\) is called a \emph{realization} of the pmGH convergence
\((\XX_i,\dist_i,\mass_i,p_i)\to(\XX,\dist,\mass,p)\).
\end{defn}
For brevity, we will identify \((\iota_i)_*\mass_i\) with \(\mass_i\) itself.
It is possible to construct a distance \(\dist_{\rm pmGH}\) on the collection (of equivalence classes) of $\RCD(K,N)$ spaces whose converging sequences are exactly those converging
in the pointed measured Gromov-Hausdorff sense. Moreover, the class of $\RCD(K,N)$ metric measure spaces is compact in the pmGH topology.
\begin{defn}[pmGH tangent]
Let $(\XX,\dist,\mass)$ be an $\RCD(K,N)$ metric measure space. For every $r>0$ and every $x\in\XX$ we define
\[
\mass_x^r:=\mass(B_r(x))^{-1}\mass.
\]
Then, for every $p\in\XX$,
\[
{\rm Tan}_p(\XX,\dist,\mass)\coloneqq\bigg\{(\YY,\dist_\YY,\mass_\YY,q)\;\bigg|\;
\exists\,r_i\searrow 0:\,(\XX,r_i^{-1}\dist,\mass_p^{r_i},p)\overset{\rm pmGH}\longrightarrow(\YY,\dist_\YY,\mass_\YY,q)\bigg\}.
\]
\end{defn}
It is known that \({\rm Tan}_p(\XX,\dist,\mass)\) is (well-defined and) non-empty.
\subsubsection{Structure results for $\RCD$ spaces}
Let us recall the definition of the regular set and some well-known structure results in the setting of $\RCD$ spaces.
\begin{defn}[Regular set]
Let \(n\in\mathbb N\) be given. Let \(\dist_e\) stands for the Euclidean distance
\(\dist_e(x,y)\coloneqq|x-y|\) on \(\RR^n\), while \(\underline{\mathcal L}^n\)
is the normalized measure
 \(\underline{\mathcal L}^n=\frac{n+1}{\omega_n}\mathcal L^n\), where $\omega_n$ is the volume of the unit ball in $\mathbb R^n$. Then the set of \emph{\(n\)-regular points} of an $\RCD(K,N)$ space
\((\XX,\dist,\mass)\) is defined as
\[
\mathcal R_n=\mathcal R_n(\XX)\coloneqq\Big\{x\in\XX\;\Big|\;{\rm Tan}_x(\XX,\dist,\mass)=\big\{(\RR^n,\dist_e,\underline{\mathcal L}^n,0)\big\}\Big\}.
\]
\end{defn}

As proven in \cite{Mondino-Naber14,GP16-2,DPMR16,KelMon16,bru2018constancy}, the following structure theorem holds.
\begin{thm}\label{thm:Structure}
Let \((\XX,\dist,\mass)\) be an \(\RCD(K,N)\) space. Then there exists a (unique) number \(n\in\mathbb N\) with \(1\leq n\leq N\), called the
\emph{essential dimension} of \((\XX,\dist,\mass)\), such that \(\mass(\XX\setminus\mathcal R_n)=0\). Moreover, the regular set
\(\mathcal R_n\) is \((\mass,n)\)-rectifiable and it holds that \(\mass\ll\mathcal H^n\).
\end{thm}
Recall that \(\mathcal R_n\) is said to be \emph{\((\mass,n)\)-rectifiable} provided there exist Borel subsets \((A_i)_{i\in\mathbb N}\) of
\(\mathcal R_n\) such that each \(A_i\) is bi-Lipschitz equivalent to a subset of \(\RR^n\) and \(\mass(\mathcal R_n\setminus\bigcup_i A_i)=0\).
\begin{defn}
Let \((\XX,\dist,\mass)\) be an \(\RCD(K,N)\) space having essential dimension \(n\). Then we define the set
\(\mathcal R_n^*=\mathcal R_n^*(\XX)\subseteq\mathcal R_n\) as
\[
\mathcal R_n^*\coloneqq\left\{x\in\mathcal R_n\;\bigg|\;\exists\Theta_n(\mass,x)\coloneqq\lim_{r\to 0}\frac{\mass(B_r(x))}{r^n}\in(0,+\infty)\right\}.
\]
\end{defn}
Notice that the set \(\mathcal R_n^*\) is Borel, see \cite[Remark 2.5]{ABPrank}. As shown in \cite[Theorem 4.1]{AHT17}, it holds that \(\mass(\XX\setminus\mathcal R_n^*)=0\).
\subsubsection{Structure results for sets of finite perimeter in $\RCD$ spaces}
Let us now recall some structure results for sets of finite perimeter in $\RCD(K,N)$ spaces. We assume the reader to be familiar with \cite{ambrosio2018rigidity, bru2019rectifiability,bru2021constancy,BGBV}.
\begin{defn}[Tangents to a set of finite perimeter]
Let \((\XX,\dist,\mass,p)\) be a pointed \(\RCD(K,N)\) space, \(E\subseteq\XX\) a set of locally finite perimeter. Then
we define \({\rm Tan}_p(\XX,\dist,\mass,E)\) as the family of all quintuplets \((\YY,\dist_\YY,\mass_\YY,q,F)\) that verify the following two conditions:
\begin{enumerate}[label=\roman*)]
\item \((\YY,\dist_\YY,\mass_\YY,q)\in{\rm Tan}_p(\XX,\dist,\mass)\),
\item \(F\subseteq\YY\) is a set of locally finite perimeter with \(\mass_\YY(F)>0\) for which the following property holds:
along a sequence \(r_i\searrow 0\) such that \((\XX,r_i^{-1}\dist,\mass_p^{r_i},p)\to(\YY,\dist_\YY,\mass_\YY,q)\) in the pmGH sense, with
realization \(\ZZ\), it holds that \(E^i\to F\) in \(L^1_{\rm loc}\) (cf.\! with \cite[Definition 3.1]{ambrosio2018rigidity}), where \(E^i\) is intended in the rescaled space \((\XX,r_i^{-1}\dist)\). If this is the case, we write
$$
(\XX,r_i^{-1}\dist,\mass_p^{r_i},p,E)\to(\YY,\dist_\YY,\mass_\YY,q,F).
$$
\end{enumerate}
\end{defn}
\begin{defn}[Reduced boundary]\label{def:ReducedBoundary}
Let \((\XX,\dist,\mass)\) be an \(\RCD(K,N)\) space. Let \(E\subseteq\XX\) be a set of locally finite perimeter.
Then we define the \emph{reduced boundary} \(\mathcal F E\subseteq\partial^* E\) of \(E\) as the set of all those
points \(x\in \mathcal R_n^*\) satisfying all the four conclusions of \Cref{zuppa} and such that 
\begin{equation}\label{eqn:FnE}
\mathrm{Tan}_x(\XX,\dist,\mass,E)=\big\{(\RR^n,\dist_e,\underline{\mathcal L}^n,0,\{x_n>0\})\big\},
\end{equation}
where \(n\in\mathbb N\), \(n\leq N\) stands for the essential dimension of \((\XX,\dist,\mass)\). We recall that the set of all points $x\in\XX$ that satisfy \eqref{eqn:FnE} was denoted by $\mathcal{F}_nE$ in \cite{ambrosio2018rigidity}. 
\end{defn}
We recall here, for the reader's convenience, \cite[Remark 2.22]{ABPrank}.
\begin{rem}\label{rem:properties_FE}
	By the proof of \cite[Corollary 4.10]{ambrosio2018rigidity}, by \cite[Corollary 3.4]{ambrosio2018rigidity}, and by the membership to $\mathcal R_n^*$,
	we see that for any $x\in\FF E$ the following hold.
	\begin{enumerate}[label=\roman*)]
		\item If $r_i\searrow 0$ is such that \begin{equation}\label{convsemplice}(\XX, r_i^{-1}\dist,\mass_x^{r_i},x)\rightarrow (\RR^n,\dist_e, \underline{\mathcal L}^n,0)\end{equation}
		in a realization $(\ZZ,\dist_\ZZ)$, then, up to not relabelled subsequences and a change of coordinates in $\RR^n$,
		$$(\XX, r_i^{-1}\dist,\mass_x^{r_i},x,E)\rightarrow (\RR^n,\dist_e, \underline{\mathcal L}^n,0,\{x_n>0\}),$$
		in the same realization $(\ZZ,\dist_\ZZ)$. Notice that, given a sequence $r_i\searrow 0$, it is always possible to find a subsequence satisfying \eqref{convsemplice}.
		\item\label{conv}  If $r_i\searrow 0$ is such that $$(\XX, r_i^{-1}\dist,\mass_x^{r_i},x,E)\rightarrow (\RR^n,\dist_e, \underline{\mathcal L}^n,0,\{x_n>0\})$$ in a realization $(\ZZ,\dist_\ZZ)$, then $|\DIFF\nchi_E|$ weakly converges to $|\DIFF\nchi_{\{x_n>0\}}|$ in duality with $C_{\mathrm{bs}}(Z)$.
		\item We have
		\begin{equation}\label{usefullimits}
			\begin{split}
				&\lim_{r\searrow 0}\frac{\mass(B_r(x))}{r^n}=\omega_n\Theta_n(\mass,x)\in (0,+\infty),\\
				&\lim_{r\searrow 0}\frac{|\DIFF\nchi_E|(B_r(x))}{r^{n-1}}={\omega_{n-1}}\Theta_n(\mass,x).
			\end{split}
		\end{equation}
	\end{enumerate}\fr
	
\end{rem}
We now recall \cite[Theorem 3.2]{ABPrank}, which follows along the techniques of \cite{bru2021constancy} and builds upon \cite{deng2020holder}. 
\begin{thm}\label{thm:const_dim_cod1}
Let \((\XX,\dist,\mass)\) be an \(\RCD(K,N)\) space having essential dimension \(n\). Then
\[
|\DIFF f|(\XX\setminus\mathcal R_n^*)=0,\quad\text{ for every }f\in\BV_{\mathrm{loc}}(\XX).
\]
\end{thm}
The following is \cite[Theorem 3.3]{ABPrank}.
\begin{thm}[Representation formula for the perimeter]\label{thm:repr_per}
Let \((\XX,\dist,\mass)\) be an \(\RCD(K,N)\) space having essential dimension \(n\). Let \(E\subseteq\XX\) be a set of locally finite perimeter. Then
\begin{equation}\label{eq:repr_per}
|\DIFF\nchi_E|=\Theta_n(\mass,\cdot)\HH^{n-1}\mres\mathcal F E.
\end{equation}
In particular, it holds that \(\Theta_{n-1}(|\DIFF\nchi_E|,x)=\Theta_n(\mass,x)\) for \(\HH^{n-1}\)-a.e. \(x\in\mathcal{F} E\).
\end{thm}

\subsubsection{Good coordinates and good splitting maps}
In this section we recall the notion of good coordinates studied in \cite{bru2021constancy} and good splitting maps introduced in \cite{ABPrank}. 
First, we recall the following, which comes from \cite[Theorem 3.13]{BGBV}, see also \cite[Theorem 2.4]{bru2019rectifiability}. We assume the reader to be familiar with the Sobolev calculus on RCD spaces and with the notion of capacitary tangent module $L^0_{\mathrm{Cap}}(TX)$. We refer to \cite[Section 2.2.1]{ABPrank} and references therein.
\begin{thm}\label{thm:DivergenceFormula}
Let $(\XX,\dist,\mass)$ be an $\RCD(K,N)$ space and let $f\in \BV(\XX)$. Then there exists a unique, up to $|\DIFF f|$-a.e.\ equality, element $\nu_f\in L^0_\capa(T\XX)$ such that $|\nu_f|=1\ |\DIFF f|$-a.e.\ and 
    $$
    \int_\XX f{\rm div}(v)\,\dd{\mass}=-\int_\XX \pi_{|\DIFF f|}(v)\cdot\nu_f\,\dd{|\DIFF f|},\quad\text{ for every $v\in\TestV(\XX)$.}
    $$
\end{thm}

In particular, if $E$ is a set of locally finite perimeter, we naturally have a unique, up to $|\DIFF \nchi_E|$-a.e.\ equality, element $\nu_E\in L^0_\capa (T\XX)$, where we understand $\nu_E=\nu_{\nchi_E}$ by locality.
\begin{defn}[Good coordinates]\label{defn:good_coordinates}
	Let \((\XX,\dist,\mass)\) be an \(\RCD(K,N)\) space of essential dimension \(n\). Let \(E\subseteq\XX\) be a set of locally finite perimeter
	and let \(y\in\mathcal F E\) be given. Then we say that an \(n\)-tuple \(u=(u^1,\ldots,u^n)\) of harmonic functions \(u^i\colon B_{r_y}(y)\to\RR\)
	is a \emph{system of good coordinates} for \(E\) at \(y\) provided the following properties are satisfied:
\begin{enumerate}[label=\roman*)]
		\item For any \(i,j=1,\ldots,n\), it holds that
		\begin{equation}\label{eqn:Eq1Good}
		\lim_{r\searrow 0}\fint_{B_r(y)}|\nabla u^i\cdot\nabla u^j-\delta_{ij}|\,\diff\mass=
		\lim_{r\searrow 0}\fint_{B_r(y)}|\nabla u^i\cdot\nabla u^j-\delta_{ij}|\,\diff|\DIFF\nchi_E|=0.
		\end{equation}
		\item For any \(i=1,\ldots,n\), it holds that
		\begin{equation}\label{eq:def_nu_in_good_coord}
			\exists\,\nu_i(y)\coloneqq\lim_{r\searrow 0}\fint_{B_r(y)}\nu_E\cdot\nabla u^i\,\diff|\DIFF\nchi_E|,
			\qquad\lim_{r\searrow 0}\fint_{B_r(y)}|\nu_i(y)-\nu_E\cdot\nabla u^i|\,\diff|\DIFF\nchi_E|=0.
		\end{equation}
		\item The resulting vector \(\nu(y)\coloneqq(\nu_1(y),\ldots,\nu_n(y))\in\RR^n\) satisfies \(|\nu(y)|=1\).
\end{enumerate}
\end{defn}
It follows from \cite[Proposition 3.6]{bru2021constancy} that good coordinates exist at \(|\DIFF\nchi_E|\)-a.e.\ \(y\in\mathcal F E\).

\begin{rem}\label{convdelta}
	Let $(\XX,\dist,\mass)$ be an $\RCD(K,N)$ space of essential dimension $n$, let $x\in\XX$ and let $u=(u^1,\dots,u^n)$ be an $n$-tuple of harmonic functions satisfying 
	$$
	\lim_{r\searrow 0}\fint_{B_r(x)}|\nabla u^i\cdot\nabla u^j-\delta_{ij}|\,\diff\mass=0.
	$$
	Given a sequence of radii $r_i\searrow 0$ such that $$(\XX,r_i^{-1}\dist,\mass_x^{r_i},x)\rightarrow(\RR^n,\dist_e,\underline{\mathcal L}^n,0)$$ and fixed a realization of such convergence, it follows from the results recalled in \cite[Section 1.2.3]{bru2019rectifiability} (see the references therein, see also \cite[(1.22)]{bru2019rectifiability},
	consequence of the improved Bochner inequality in \cite{Han14}) that, up to extracting a not relabelled subsequence, the functions in
	$$\{r_i^{-1} (u^j-u^j(x))\}_i\qquad\text{for }j=1,\dots,n $$
	converge locally uniformly to orthogonal coordinate functions of $\RR^n$.\fr
\end{rem}

Let us now recall the notion of $\delta$-splitting map. We follow the presentation in \cite[Definition 3.4]{bru2019rectifiability}, see also \cite[Section 2.2.3]{ABPrank}.
\begin{defn}[Splitting map]\label{def:Splitting}
Let \((\XX,\dist,\mass)\) be an \(\RCD(K,N)\) space. Let \(y\in\XX\), \(k\in\mathbb N\), and \(r_y,\delta>0\) be given.
Then a map \(u=(u_1,\ldots,u_k)\colon B_{r_y}(y)\to\RR^k\) is a \emph{\(\delta\)-splitting map} if
the following three properties hold:
\begin{enumerate}[label=\roman*)]
\item \(u_i\) is harmonic, meaning that, for every $i=1,\dots,k$, $u_i\in D(\Delta,B_{r_y}(y))$ and $\Delta u_i=0$; and moreover $u_i$ is $C_{K,N}$-Lipschitz for every $i=1,\dots,k$,
\item \(r_y^2\dashint_{B_{r_y}(y)}|{\rm Hess}(u_i)|^2\,\dd\mass\leq\delta\) for every \(i=1,\ldots,k\),
\item \(\dashint_{B_{r_y}(y)}|\nabla u_i\cdot\nabla u_j-\delta_{ij}|\,\dd\mass\leq\delta\) for every \(i,j=1,\ldots,k\).
\end{enumerate}
\end{defn}

For what follows, recall the definition of good splitting map, compare with \cite[Definition 2.28]{ABPrank} and \Cref{goodsplitting}.

\begin{rem}
Let $u$ be a good splitting map on $D\subseteq B_{r_y}(y)$. Due to \cite[Remark 2.10]{bru2021constancy}, for every $x\in B_{r_y}(y)$ there exists a Borel matrix $M(x)=\{M(x)_{i,j}\}_{i,j=1,\dots,n}\in\RR^{n\times n}$ satisfying 
	\begin{equation}\label{Mmatrix}
		\lim_{s\searrow 0}\dashint_{B_s(x)}|\nabla u^i\,\cdot\,\nabla u^j-M(x)_{i,j}|\,\dd\mass=0\qquad\text{for every }i,j=1,\dots,n.
	\end{equation}
Then, from item iii) of \Cref{def:Splitting} and since $\eta<n^{-1}$, we have that for every $x\in D$
\begin{equation}\label{cdskcee}
    |M(x)_{i,j}-\delta_{ij}|\leq \eta < n^{-1}.
\end{equation}
Hence, by applying the Gram--Schmidt orthogonalization algorithm to $\{\nabla u^i(x)\}_{i=1,\dots,n}$ for every $x\in D$, we find a matrix-valued function $A\in L^{\infty}(D;\mathbb R^{n\times n})$ such that, for every $x\in D$,
\begin{equation}\label{Amatrix}
	A(x) M(x)A(x)^T=\rm Id.\
\end{equation}
The membership $A\in L^{\infty}(D;\mathbb R^{n\times n})$ is due to \eqref{cdskcee}.
\fr
\end{rem}

Let  $(\XX,\dist,\mass)$ be an $\RCD(K,N)$ space of essential dimension $n$, and let $f\in\BV(\XX)$.
Recall that \cite[Theorem 5.1]{Ambrosio-Pinamonti-Speight15} and its proof (compare with \cite[Proposition 4.2]{ambmirpal04} and \cite[Proposition 2.13]{ABPrank}) yield that $\GG_f$ has locally finite perimeter and 
\begin{equation}\label{diffandper}
	|\DIFF f|\le  \pi_*^1|\DIFF\nchi_{\GG_f}|\le|\DIFF f|+\mass.
\end{equation}
Also, by \cite[Theorem 3.3]{ABPrank}, it holds that 
\begin{equation}\label{PerHn}
 \HH^{n}\mres \partial^*\GG_f\ll |\DIFF\nchi_{\GG_f}|,
\end{equation}
so that, taking into account \Cref{lem:char_G_f} and \Cref{thm:Structure}, we see that 
\begin{equation}\label{massandpi}
	\mass\ll\HH^n\mres \mathcal R_n\ll\pi_*^1|\DIFF\nchi_{\GG_f}|.
\end{equation}

Before going on, we stress a couple of remarks.
\begin{rem}
	 Taking into account \Cref{def:Normals}, it holds that $\nu_{\GG_f}^u$ is well-defined at $(x,\bar f(x))$ for $|\DIFF f|$-a.e.\ $x\in D\setminus J_f$ and $\mass$-a.e.\ $x\in D\setminus J_f$, as a consequence of \eqref{diffandper} and \eqref{massandpi}, respectively, together with \Cref{lem:char_G_f}.
\fr
\end{rem}

\begin{rem}\label{remarkfrequenteAM}
We isolate here an argument which will frequently appear during the paper, and that is essentially contained in \cite[Proposition 3.6]{bru2021constancy}. Let $(\XX,\dist,\mass)$ be an $\RCD(K,N)$ space of essential dimension $n$, and let $E\subseteq X$ be a set of locally finite perimeter.

Let $u: B_{2r_y}(y)\to\mathbb R^n$ be a good splitting map on $D\subseteq B_{r_y}(y)$. We claim that, for $|\DIFF \nchi_E|$-almost every point $x\in \mathcal{F}E\cap D$, the function $v:=A(x)u:B_{2r_{y}}(y)\to\mathbb R^n$ is a system of good coordinates for $E$ at $x$, where the matrix-valued function $A$ is defined as in \eqref{Amatrix}. In addition, if $\nu_E^u:B_{2r_y}(y)\to \mathbb R^n$ is the $|\DIFF\nchi_E|$-measurable map 
\[
\nu_E^u(x):=((\nu_E\cdot\nabla u^1)(x),\ldots,(\nu_E\cdot\nabla u^n)(x)),
\]
then the normal $\nu_E^v$ associated to the system of good coordinates $v$ for $E$ at $x$ (see item ii) of \Cref{defn:good_coordinates}) is 
\[
\nu_E^v=A(x)\nu_E^u.
\]

Indeed, let us fix $x\in \mathcal{F}E\cap D$ that is, for every $i,j=1,\dots,n$, a Lebesgue point of all the functions $\nabla u^i \cdot \nabla u^j$, $\nu_E\cdot A\nabla u^i$, $\nu_E\cdot \nabla u^i$, and $A$, with respect to the asymptotically doubling measure $|\DIFF\nchi_E|$. Let us denote $v^i:=A(x) u^i$. We aim at showing that $(v^i)_{i=1,\dots,n}:B_{2r_y}(y)\to\mathbb R^n$ are good coordinates for $E$ at $x$.

First, $v^i$ are harmonic. Second, by the very definition of $A$ and $M$, see \eqref{Amatrix} and \eqref{Mmatrix}, and by the fact that $x$ is a Lebesgue point of $\nabla u^i \cdot \nabla u^j$ with respect to $|D\nchi_E|$, we have the two equalities in \eqref{eqn:Eq1Good} at $x$ with $v^i$. Third, by denoting $a_i(x)$ the Lebesgue value of $\nu_E\cdot A\nabla u^i$ at $x$ with respect to $|D\nchi_E|$, and since by definition $(A\nabla u^i)\cdot (A\nabla u^j)=\delta_{ij}$ everywhere on $D$, we conclude that $\nu_E=\sum_{i=1}^n a_i(A\nabla u^i)$ holds $|\DIFF \nchi_E|$-almost everywhere on $D$. Now, since $|\nu_E|=1$ in $L^2_E(T\XX)$, and since $(A\nabla u^i)$ are pointwise orthonormal on $D$, we conclude that the vector $(a_i)_{i=1,\dots,n}$ has norm 1. Hence, by finally taking into account that $x$ is also a Lebesgue point of $\nu_E\cdot\nabla u^i$ and $A$ with respect to $|\DIFF \nchi_E|$, \eqref{eq:def_nu_in_good_coord} and item ii) in \Cref{defn:good_coordinates} hold. How the normal transforms is clear from \eqref{eq:def_nu_in_good_coord}. Thus the claim is proved.
\fr
\end{rem}

\section{Main results}\label{sec:Main}

In this section we are going to prove the main results of this note, i.e., \Cref{thmmain1} and \Cref{thmmain2}. First, we start with some auxiliary results.
\subsection{Auxiliary results}
For this section we fix an $\RCD(K,N)$ space of essential dimension $n$ $(\XX,\dist,\mass)$ and $f\in\BV(\XX)$.
We fix also $u$, a good splitting map on $D\subseteq B_{r_x}(x)$ for some $x\in \XX$ and $r_x>0$.

The following proposition can be proved exactly as \cite[Proposition 3.6]{ABPrank}. 
Recall the definition of the reduced boundary in use in this note, see \Cref{def:ReducedBoundary}.
\begin{prop}\label{prop:def_set_C_f}
	In the setting above, there exists a Borel set \(D_f\subseteq D\) satisfying the following properties:	\begin{enumerate}[label=\roman*)]
		\item $|\DIFF f|^c(D\setminus D_f)=0$ and \(\mass(D\setminus D_f)=0\).
		\item $|\DIFF\nchi_{\GG_f}|((D\setminus (D_f\cup J_f))\times \RR)=0$.
		\item \(D_f\subseteq\mathcal R_n^*(\XX)\setminus J_f\) and \(\mathcal F\GG_f\cap(D_f\times\RR)
		=({\rm id}_\XX,\bar f)(D_f)\).
		\item Given any \(x\in D_f\), for
		\(A(x)\in\RR^{n\times n}\) as in \eqref{Amatrix}, we have that
		\((A(x) u,\pi^2)\) is a system of good coordinates for \(\GG_f\) at \((x,\bar f(x))\).
		\item If \(v=(v^1,\ldots,v^{n+1})\colon B_{r_x}(x,\bar f(x))\to\RR^{n+1}\) is a system of good coordinates for \(\GG_f\) at \((x,\bar f(x))\)
		for some \(x\in D_f\) and the coordinates \((x_i)\) on the (Euclidean) tangent space to \(\XX\times\RR\) at \((x,\bar f(x))\) are chosen so that the maps
		\((u^i)\) converge to \((x_i)\colon\RR^{n+1}\to\RR^{n+1}\) (when properly rescaled, see \Cref{convdelta}), then the blow-up of \(\GG_f\) at \((x,\bar f(x))\) can be written as
		\[
		H\coloneqq\big\{y\in\RR^{n+1}\;\big|\;y\cdot\nu(x,\bar f(x))\geq 0\big\},
		\]
		where the unit vector \(\nu(x,\bar f(x))\coloneqq\big(\nu_1(x,\bar f(x)),\ldots,\nu_{n+1}(x,\bar f(x))\big)\) is given by \eqref{eq:def_nu_in_good_coord} for $v$.
	\end{enumerate}
\end{prop}
\begin{proof}
	The proof is \cite[Proposition 3.6]{ABPrank}, up to the fact that we replace the definition in \cite[(3.11)]{ABPrank} with
	\begin{equation}\notag
		D_f\coloneqq \XX_f\cap\left(\mathcal{R}_n^*(\XX)\setminus J_f\right)\cap D
		\cap \pi^1(\mathcal A\cap \mathcal{T}\cap\mathcal{S}\cap\mathcal{F}\mathcal{G}_f)\cap\mathcal{D},
	\end{equation}
where we kept the notation as in the proof of \cite[Proposition 3.6]{ABPrank}.
Item i) is obtained recalling \eqref{diffandper} and \eqref{massandpi}.  Item ii)  follows from the definition of $D_f$ again together with \eqref{diffandper}.
The other items are proved exactly as in the reference. In particular, for item iv), see \Cref{remarkfrequenteAM}.
\end{proof}
Even though by definition $D_f\cap J_f=\varnothing$, we sometimes consider $D_f\setminus J_f$ to remind this fact.
In our proofs, we will implicitly take as a representative of $\nu_{\GG_f}^u$ (see \Cref{def:Normals}) its Lebesgue representative with respect to the asymptotically doubling measure $|\DIFF\nchi_{\GG_f}|$.  This will not make any difference in the end, due to the nature of the statements, but will allow us to exploit item v) of \Cref{prop:def_set_C_f}.
\begin{lem}\label{normzero}
Set 
$$ \hat{C}_f\defeq\left\{x\in D: \big(\nu_{\GG_f}^u(x,\bar f(x))\big)_{n+1}=0\right\}.$$
Then $\mass(\hat C_f)=0$.
\end{lem}
\begin{proof}
By \Cref{thm:Structure} we only have to show that $(\HH^n\mres\mathcal R_n)(\hat C_f\cap D_f)=0$. Then, by \cite[Theorem 2.4.3]{AmbrosioTilli04} it is enough to show that 
\begin{equation}\notag
	\liminf_{r\searrow 0} \frac{|\DIFF f|(B_r(x))}{r^n}=+\infty 
\end{equation}
for $\HH^n$-a.e.\ $x\in\hat C_f\cap D_f$. Therefore, by \eqref{diffandper}, and by taking into account that $D_f\subseteq\mathcal{R}_n^*$, it is enough to show that
\begin{equation}\label{csdncas}
	\liminf_{r\searrow 0} \frac{|\DIFF \nchi_{\GG_f}|(B_r(x)\times\RR)}{r^n}=+\infty 
\end{equation}
for $\HH^n$-a.e.\ $x\in\hat C_f\cap D_f$.

The conclusion then follows from a blow-up argument.  Now we follow the first part of the proof of \cite[
Theorem 3.7]{ABPrank}, we sketch the argument. Take $x\in\hat C_f\cap D_f$, let $p\defeq(x,\bar{f}(x))$, take a sequence $\{r_i\}_i\subseteq (0,\infty)$, $r_i\searrow 0$.  We use repeatedly the membership $p\in\FF\GG_f$ and the implied properties as in \Cref{rem:properties_FE}.
We have that, up to subsequences,
\[
(\XX,r_i^{-1}\dist,\mass_x^{r_i},x)\to(\RR^n,\dist_e,\underline{\mathcal L}^n,0),\quad\text{ in the pmGH topology.}
\]
Let \((\ZZ,\dist_\ZZ)\) be a realization of such convergence. Then \((\ZZ\times\RR,\dist_\ZZ\times\dist_e)\) is a realization of
\[
\big(\XX\times\RR,r_i^{-1}(\dist\times\dist_e),(\mass\otimes\mathcal L^1)_p^{r_i},p,\GG_f\big)\to(\RR^{n+1},\dist_e,\underline{\mathcal L}^{n+1},0,H),
\]
where \(H\subseteq\RR^{n+1}\) is a halfspace, and where we took a non-relabelled subsequence. We also know, as $p\in\FF\GG_f$, that 
the rescaled perimeters \(|\DIFF\nchi_{\GG_f}|\) weakly converge, up to some dimensional constant, to \(\HH^n\mres\partial H\) in duality with \(C_{\rm bs}(\ZZ\times\RR)\).
Moreover, by the definition of $\hat C_f$ together with the item $v)$ of \Cref{prop:def_set_C_f}, and the fact that the last coordinates of $\nu_{\mathcal{G}_f}^u$ and $\nu_{\mathcal{G}_f}^v$ are equal,
we have that  $H$ can be written as $H'\times\RR$, for some $H'\subseteq\RR^n$ halfspace. Then the claim follows from weak convergence of measures, taking into account also item iii) of \Cref{rem:properties_FE}.
\end{proof}

Taking into account the proof of \cite[Theorem 3.7]{ABPrank}, and \Cref{prop:def_set_C_f}, we have that $|\DIFF f|^c$ is concentrated on $\hat{C}_f$, so that $|\DIFF f|^c=|\DIFF f|\mres \hat{C}_f$.
\begin{lem}
\label{csmkc}
It holds that
$$
\big(\nu_{\GG_f}^{ u}(x,\bar f(x))\big)_{1,\dots,n}=\sqrt{1-\Big(\nu_{\GG_f}^u(x,\bar f(x))\Big)_{n+1}^2}\,\nu_f^{u}(x),\quad\text{ for }|\DIFF f|\text{-a.e.\ }x\in D\setminus J_f
$$
and 
$$
\big(\nu_{\GG_f}^{ u}(x,\bar f(x))\big)_{n+1}\le 0,\quad\text{ for }|\DIFF f|\text{-a.e.\ }x\in D\setminus J_f.
$$
\end{lem}
\begin{proof}
The proof follows the lines of the proof of \cite[Lemma 3.8]{ABPrank}, but the conclusion is slightly different. We sketch here the argument.

By \Cref{prop:def_set_C_f} together with \eqref{PerHn}, \Cref{coarea}, and \cite[Lemma 3.27]{BGBV}, it is enough to show that for a.e.\ $t\in\RR$,
$$
\big(\nu_{\GG_f}^{u}(x,\bar f(x))\big)_{1,\dots,n}=\sqrt{1-\Big(\nu_{\GG_f}^u(x,\bar f(x))\Big)_{n+1}^2}\,\nu_{E_t}^{u}(x),\quad\text{ for }\HH^{n-1}\text{-a.e.\ }x\in D_f\cap\FF E_t,$$
where $E_t\defeq\{f>t\}$. We fix $t$ such that $E_t$ is a set of finite perimeter.

Take $x\in D_f\cap\FF E_t$ such that $A(x)u$ is a system of good coordinates for $E_t$ at $x$ and such that the conclusion of \cite[Proposition 4.8]{bru2021constancy} holds at $x$. Notice that $\HH^{n-1}$-a.e.\  $x\in D_f\cap\FF E_t$ satisfies the previous two properties, as a consequence of a Lebesgue point argument as the one in item iii) of the proof of \cite[Proposition 3.6]{ABPrank}. We can also assume that if $\nu=\nu(x)$ and $\mu=\mu(x,\bar f(x))$ are given by \eqref{eq:def_nu_in_good_coord} for $A(x) u$ and $(A(x)u,\pi^2)$ respectively (for $E_t$ and $\GG_f$ respectively), it holds that 
\begin{equation}\label{njld}
\nu=A(x)\nu_{E_t}^u(x)\qquad \mu=\Big(A(x)\big(\nu_{\GG_f}^u(x,\bar f(x))\big)_{1,\dots,n},\big(\nu_{\GG_f}^u(x,\bar f(x))\big)_{n+1}\Big),
\end{equation}
compare with \Cref{remarkfrequenteAM}, and that $|\nu|=|\mu|=1$.
Then, following \cite[Lemma 3.8]{ABPrank} (with the same notation), we have that 
$$
H'\times (-\infty,0)\subseteq H \cap\{(y,s)\in\RR^n\times \RR:s<0\},
$$
where 
$$
H'=\{y\in\RR^n:y\,\cdot\, \nu\ge 0\},\qquad H=\{z\in\RR^{n+1}:z\,\cdot\, \mu\ge 0\}.
$$
Therefore, $\mu=(\alpha\nu,\mu_{n+1})$, for some $\alpha\in [0,1]$. Now we have that 
$$
1=|\mu|^2=\alpha^2|\nu|^2+\mu_{n+1}^2=\alpha^2+\mu_{n+1}^2,
$$
so that we conclude recalling \eqref{njld} and the fact that $A(x)$ is invertible.  
\end{proof}
\begin{lem}\label{cdskcasn}
It holds that 
$$\dv{\pi^1_*|\DIFF\nchi_{\GG_f}|}{\mass}(x)=-\big((\nu_{\GG_f}^{ u})_{n+1}(x,\bar f(x))\big)^{-1},\quad\text{ for $\mass$-a.e.\ $x\in D\setminus (J_f\cup  C_f)$}.$$
\end{lem}
\begin{proof}
	Recalling \Cref{prop:def_set_C_f}, we can reduce ourselves to show the conclusion only for $\mass$-a.e.\ $x\in D_f\setminus {( C_f\cup J_f)}$. 
 
	By \Cref{normzero}, we know that\begin{equation}\label{vfondsanj}
		 (\nu_{\GG_f}^{{ u}})_{n+1}(x,\bar f(x))\ne 0,\quad\text{ for $\mass$-a.e.\ $x\in D_f\setminus  {(C_f\cup J_f)}$}.
	\end{equation} 
	We prove now that in fact
 \begin{equation}\label{mkmdsm}
	(\nu_{\GG_f}^{{ u}})_{n+1}(x,\bar f(x))< 0,\quad\text{ for $\mass$-a.e.\ $x\in D_f\setminus  {( C_f\cup J_f)}$}.
\end{equation}
Fix $x\in D_f\setminus C_f$ satisfying \eqref{vfondsanj}, let $p\defeq(x,\bar f(x))$ and take $\{r_i\}_i\subseteq(0,\infty)$ with $r_i\searrow 0$. Up to subsequences, we have that 
\[
\big(\XX\times\RR,r_i^{-1}(\dist\times\dist_e),(\mass\otimes\mathcal L^1)_p^{r_i},p,\GG_f\big)\to(\RR^{n+1},\dist_e,\underline{\mathcal L}^{n+1},0,H),
\]
in a realization $(\ZZ\times\RR,\dist_{\ZZ\times\RR})$,
for some halfspace $H$. Arguing as in the proof of \Cref{csmkc}, it holds that $H=\{z\in \RR^{n+1}:z\,\cdot\,\mu\ge 0\}$,
for $$\mu=\Big(A(x)\big(\nu_{\GG_f}^u(x,\bar f(x))\big)_{1,\dots,n},\big(\nu_{\GG_f}^u(x,\bar f(x))\big)_{n+1}\Big).$$
Let $\pm B_\epsilon\defeq B_\epsilon^\ZZ(0_{\RR^n})\times B_\epsilon^{\RR}(\pm 1_\RR)\subseteq\ZZ\times\RR$.
Take $\epsilon>0$ small enough so that $(\pm B_{ \epsilon})\cap \partial H\ne\varnothing$. Such $\epsilon$ exists by \eqref{vfondsanj}. Now we compute, by convergence in $L^1_{\rm loc}$,
\begin{align*}
&\underline{\mathcal L}^{n+1}(H\cap(-B_\epsilon))-\underline{\mathcal L}^{n+1}(H\cap B_\epsilon)=\lim_{i\to+\infty} \big((\mass\otimes \mathcal L^1)_p^{r_i}(\GG_f\cap(-B_\epsilon))-(\mass\otimes \mathcal L^1)_p^{r_i}(\GG_f\cap B_\epsilon)\big).
\end{align*}
By Fubini's theorem,
\begin{align*}
&(\mass\otimes \mathcal L^1)_p^{r_i}(\GG_f\cap(-B_\epsilon))-(\mass\otimes \mathcal L^1)_p^{r_i}(\GG_f\cap B_\epsilon)\\&\qquad=\frac{1}{(\mass\otimes \mathcal L^1)(B_{r_i}(p))}\int_{B_\epsilon^\ZZ(0_{\RR^n})} \bigg(\mathcal H^1\big((\{x\}\times (-r_i-\epsilon,-r_i+\epsilon))\cap \GG_f\big)\\&\qquad\qquad\qquad\qquad\qquad\qquad\qquad\qquad\qquad-\mathcal H^1\big((\{x\}\times (r_i-\epsilon,r_i+\epsilon))\cap \GG_f\big) \bigg)\dd\mass(x)
\\&\qquad\ge 0,
\end{align*}
where at the second member we have the image measure of $\mass$ corresponding to the $i^{\rm th}$ rescaling. Therefore  
$\underline{\mathcal L}^{n+1}(H\cap(-B_\epsilon))-\underline{\mathcal L}^{n+1}(H\cap B_\epsilon)\ge 0$, so that \eqref{mkmdsm} follows, taking into account also \eqref{vfondsanj} and the defining expression for $H$.
	
	Fix a ball $\bar B\subseteq\XX$.
By  \eqref{vfondsanj} and the proof of \cite[Theorem 3.7]{ABPrank} (see in particular \cite[(3.20) and (3.24)]{ABPrank}) we know that for every $\epsilon>0$ there exist $\gamma=\gamma(\epsilon)\in (0,\infty)$ and a set $F_{\epsilon}\subseteq \FF\GG_f$ such that 
$$|\DIFF\nchi_{\GG_f}|\Big(\big( ((\bar B\cap D_f)\setminus  C_f)\times\RR\big)\setminus F_\epsilon\Big)<\epsilon$$ 
and such that, for every $p=(x,\bar f(x))\in F_\epsilon$, there exists $r_0=r_0(p,\epsilon)$ satisfying 

\begin{equation}\label{csdknac}
	F_\epsilon\cap (B_r(x)\times\RR)\subseteq C_\gamma(x,\bar f(x)),\qquad\text{if }r<r_0,
\end{equation}
where $$
C_\gamma(x,t)\defeq\big\{(y,s)\in\XX\times\RR: \gamma\dist(y,x)\ge |s-t|\big\}.
$$
We can and will assume that ${F}_\epsilon$ is made of points of density $1$ with respect to $|\DIFF\nchi_{\GG_f}|$. This will ensure that for $p\in{F}_\epsilon$, the rescaled measures $|\DIFF\nchi_{\GG_f}|\mres F_\epsilon$ and $|\DIFF\nchi_{\GG_f}|$ around $p$ have the same weak limit.
Let $\hat F_\epsilon\defeq\pi^1 (F_\epsilon)$. 
By  a blow-up argument (taking into account \Cref{rem:properties_FE}), for $\mass$-a.e.\ $x\in \big(( \bar B\cap D_f)\setminus C_f\big)\cap \hat F_\epsilon$,
$$
\lim_{r\searrow 0}\frac{(|\DIFF\nchi_{\GG_f}|\mres F_\epsilon)(B_r(x)\times\RR)}{ r^n}=\omega_n\Theta_n(\mass,x)\big|\big(\nu_{\GG_f}^{u}(x,\bar f(x))\big)_{n+1}\big|^{-1}.
$$
Here we exploited \eqref{csdknac} and the fact that $\Theta_n(\mass,x)=\Theta_{n+1}(\mass\otimes\mathcal L^1,(x,t))$, as an easy computation shows.
Now notice that the left-hand side of the equation above reads as
$$
\dv{(\pi^1_*|\DIFF\nchi_{\GG_f}|)\mres \hat F_\epsilon}{\mass}(x)\omega_n\Theta_n(\mass,x)
$$
and, by \eqref{mkmdsm}, the right-hand side reads as $$\omega_n\Theta_n(\mass,x)\big(-\nu_{\GG_f}^{u}(x,\bar f(x))\big)_{n+1}^{-1}.$$
Now we conclude recalling \eqref{massandpi} and the arbitrariness of $\bar B$.
\end{proof}

The first part of the following lemma can be proved also exploiting \cite[Theorem 5.1]{Ambrosio-Pinamonti-Speight15}. Nevertheless, we give a different proof, tailored to this setting and more in the spirit of this paper.
\begin{lem}[Area formula]\label{area}
It holds that
\begin{equation}\notag
\dv{\pi^1_*|\DIFF\nchi_{\GG_f}|}{\mass}(x)=\sqrt{g_f(x)^2+1},\quad\text{ for $\mass$-a.e.\ $x\in D\setminus( C_f\cup J_f)$}.
\end{equation} 
\end{lem}
\begin{proof}
Recall that by \Cref{cdskcasn}, for $\mass$-a.e.\ $x\in D_f\setminus (J_f\cup  C_f)$ it holds that $\big(\nu_{\GG_f}^u(x,\bar f(x))\big)_{n+1}<0$.

We start from the case $f\in \BV(\XX)\cap\LIP(\XX)$. First, recall \cite[Proposition 6.3]{AH16} and \cite{Cheeger00}, which imply that $|\DIFF f|=(\lip f)\mass$.
Also, by \Cref{prop:def_set_C_f}, we reduce ourselves to show the claim for $\mass$-a.e.\ $x\in D_f\setminus ( C_f\cup J_f)$. Take then $x\in D_f\setminus ( C_f\cup J_f)$  such that $p\defeq(x, f(x))$ is a Lebesgue point for $\nu_{\GG_f}^u$ with respect to $|\DIFF\nchi_{\GG_f}|$. This choice can be made $\mass$-a.e. by \eqref{massandpi}.  

We take $\{x_i\}\subseteq\XX$ with $x_i\rightarrow x$ and 
$$
\lim_{i\to +\infty} \frac{f(x_i)-f(x)}{\dist(x_i,x)}=\pm\lip f(x).
$$
Set $r_i\defeq \dist(x,x_i)$, and notice that we can, and will, assume that $r_i\searrow 0$. Therefore, up to subsequences, we have that 
$$\big(\XX\times\RR,r_i^{-1}(\dist\times\dist_e),(\mass\otimes\mathcal L^1)_p^{r_i},p,\GG_f\big)\to(\RR^{n+1},\dist_e,\underline{\mathcal L}^{n+1},0,H),$$
where $H$ is the halfspace 
$$
H\defeq\big\{y\in\RR^{n+1}:y\,\cdot\,\nu_{\GG_f}^{{v}}(p)\ge 0\big\},
$$
for $v\defeq(A(x)u,\pi^2)$, see \Cref{prop:def_set_C_f} and \Cref{remarkfrequenteAM}. 
We assume that this convergence is realized in a proper metric space  $(\ZZ\times\RR,\dist_{\ZZ\times \RR})$ and, up to taking a non-relabelled subsequence, we assume that the rescaled perimeters $|\DIFF\nchi_{\GG_f}|$ weakly converge to $\frac{1}{\omega_{n+1}}\HH^n\mres\partial H$ in duality with $C_{\rm bs}(\ZZ\times\RR)$. Therefore, identifying $(x_i,f(x_i))$ with the corresponding point with respect to the $i^{\rm th}$ isometric embedding, we have that, up to a non-relabelled subsequence, $(x_i,f(x_i))\rightarrow \bar q\defeq(\bar z, \pm\lip f(x))\in\RR^{n+1}$ with respect to  $\dist_{\ZZ\times\RR}$, where $\dist_e(\bar z,0)=1$.  Therefore, if we  show that $\bar q\in\partial H$, it will follow  that
\begin{equation}\label{eqn:Ineq1}
    \big(-\nu_{\GG_f}^{ u}(x,\bar f(x)\big)_{n+1}^{-1}\ge \sqrt{\lip f(x)^2+1}.
\end{equation}

Take $\bar q'=(\bar z,t)$ such that $\bar{q}'\in\partial H$. The claim will be proved by showing that $\bar q=\bar q'$. By weak convergence of measures and \Cref{lem:char_G_f}, we find a sequence of points $\{(x_i',f(x_i'))\}_i$ with $(x_i',f(x_i'))\rightarrow \bar q'$ in $\ZZ\times\RR$, where we identified $(x_i',f(x_i'))$ with the corresponding point with respect to the $i^{\rm th}$ isometric embedding.
Now we compute, if $L$ is the global Lipschitz  constant of $f$,
\begin{align*}
	|\pm\lip f(x)-t|&=\lim_{i\to+\infty}\frac{ |f(x_i)-f(x_i')|}{r_i}\le \limsup_{i\to+\infty}L\frac{\dist(x_i,x_i')}{r_i}=\limsup_{i\to+\infty}L \dist_\ZZ (x_i,x_i')\\&\le\limsup_{i\to+\infty} L(\dist_\ZZ(x_i,\bar z)+\dist_\ZZ(x_i',\bar z))=0.
\end{align*}

Now we show the reverse inequality in \eqref{eqn:Ineq1}. Take $\bar q\defeq(\bar z,t)\in\partial H$ with $\dist_e(\bar z,0)=1$. As before, we find $(x_i,f(x_i))\rightarrow \bar q$ in $\ZZ\times\RR$.
But then
\begin{align*}
|t|&= \lim_{i\to+\infty}\frac{ |f(x_i)-f(x)|}{r_i}=\lim_{i\to+\infty}\frac{ |f(x_i)-f(x)|}{\dist(x_i,x)}\frac{\dist(x_i,x)}{r_i}\\&\le \limsup_{i\to+\infty} \frac{ |f(x_i)-f(x)|}{\dist(x_i,x)}\limsup_{i\to+\infty} \dist_{\ZZ}(x_i,x)\le\lip f(x)\dist_e(\bar z,0)=\lip f(x).
\end{align*}
This easily implies, by the arbitrariness of $\bar q$,
that 
$$\big(-\nu_{\GG_f}^{ u}(x,\bar f(x))\big)_{n+1}^{-1}\le \sqrt{\lip f(x)^2+1}.$$

Now we pass to the general case.
	Take $\epsilon>0$ and, by \cite[Proposition 4.3]{kinkorshatuo}, take $h\in\BV(\XX)\cap \LIP(\XX)$ with $\mass(\{h\ne f\})<\epsilon$. Recall \Cref{prop:def_set_C_f} and call $D_\epsilon\defeq(D_f\cap D_h\cap\{h=f\})\setminus C_f$. It will be enough to prove the claim for $\mass$-a.e.\ $x\in D_\epsilon$. Notice that by \cite[Proposition 3.7]{kinkorshatuo}, $|\DIFF (f-h)|(D_\epsilon)=0$, in particular, $g_f=\lip h\ \mass$-a.e.\ on $D_\epsilon$. 
	
 Now notice that for $\mass$-a.e.\ $x\in D_\epsilon$, it holds that $\XX\setminus\partial^*\GG_f$ is of $n$-density $0$ for $|\DIFF\nchi_{\GG_{h}}|$ at $(x,h (x))$, by \eqref{diffandper}. 
 Indeed,
 \begin{align*}
     \frac{|\DIFF\nchi_{\GG_h}|(B_r(x,h(x))\setminus\partial^* \GG_f)}{r^n}&=    \frac{|\DIFF\nchi_{\GG_h}|\big(\{(y,t)\in B_r(x,h(x)): h(y)\ne \bar f(y)\}\big)}{r^n}\\&\le \frac{(\pi_*^1|\DIFF\nchi_{\GG_h}|)(B_r(x)\cap \{h\ne\bar f\})}{r^n}\le \frac{(|\DIFF h|+\mass )(B_r(x)\cap \{h\ne f\})}{r^n},
 \end{align*}
 whence the conclusion at density $0$ points of $\{h\ne f\}$ follows, taking into account $|\DIFF h|\ll\mass$ and the fact that $\mass$ is concentrated on $\mathcal R_n^*$.
 At such points, $|\DIFF\nchi_{\GG_h}|$ and $|\DIFF\nchi_{\GG_f}|\wedge |\DIFF\nchi_{\GG_h}|$, properly rescaled, have the same weak limit. Hence, for $\mass$-a.e.\ $x\in D_\epsilon$, the blow-ups of $\GG_f$ and $\GG_h$ coincide at $(x, h(x))$, by a monotonicity argument (use also the last conclusion of \Cref{csmkc}). Now we use item v) of \Cref{prop:def_set_C_f} together with \Cref{remarkfrequenteAM} 
 to deduce that $\nu_{\GG_h}^u (x,h (x))=\nu_{\GG_f}^u (x,\bar f (x))$ holds for $\mass$-a.e.\ $x\in D_\epsilon$.
	Now, the claim follows from what proved in the first part of the proof.
\end{proof}

\begin{lem}\label{negl}
It holds that
$$\HH^n\big(\big\{(x,t):x\in J_f,t=f^\vee(x)\big\}\big)=\HH^n\big(\big\{(x,t):x\in J_f,t=f^\wedge(x)\big\}\big)=0.$$
\end{lem}
\begin{proof}
We only prove that $$\HH^n\big(\{(x,t):x\in J_f,t=f^\vee(x)\}\big)=0,$$
the other statement being analogous. Also, we can reduce ourselves to prove that
 $$
 \HH^n\big(\big\{(x,t):x\in K,t=f^\vee(x)\big\}\big)=0,
 $$
where $K\subseteq J_f$ is a compact set with $\HH^{n-1}(K)<\infty$ and $f^\vee_{|K}:K\rightarrow\RR$ is uniformly continuous.
Indeed, notice first that $\HH^n$ is $\sigma$-finite on $$\{(x,t):x\in J_f,t=f^\vee(x)\}\subseteq J_f\times\RR,$$
as $\HH^{n-1}$ is $\sigma$-finite on $J_f$. Then we can reduce ourselves to consider $\HH^n\mres \tilde K$ with $\tilde K\subseteq \{(x,t):x\in J_f,t=f^\vee(x)\}$ compact such that $K\defeq\pi^1(\tilde K)$ satisfies $\HH^{n-1}( K)<\infty$. The continuity of $f^\vee_{|K}$ comes from the fact that it is the inverse of the continuous map $\pi^1$ defined on a compact set into a Hausdorff space.

Now we conclude with a covering argument. Let $\epsilon>0$. Let also $\delta\in (0,\epsilon)$ be such that $|f^{\vee}(x)-f^\vee(y)|<\epsilon$ if $x,y\in K$ are such that $\dist(x,y)<\delta$. Now we find a sequence of balls $\{B_{r_i}(x_i)\}_i$ with  $r_i<\delta$,
$K\subseteq\bigcup_i B_{r_i}(x_i)$ and $\sum_i \omega_{n-1}{r_i}^{n-1}\le \HH^{n-1}(K)+\epsilon$. Now notice that $\tilde K\cap (B_{r_i}(x_i)\times \RR)\subseteq B_{r_i}(x_i)\times (f^\vee(x_i)-\epsilon,f^\vee(x_i)+\epsilon)$.
It isx easy to show that, as $r_i<\delta$,
$$
\HH^n_{\sqrt{2}\delta}\big( B_{r_i}(x_i)\times(f^\vee(x_i)-\epsilon,f^\vee(x_i)+\epsilon)\big)\le \omega_{n}(\sqrt 2r_i)^n(2\epsilon/r_i+1).
$$
Therefore, if $C$ denotes a constant that may vary from line to line,
\begin{align*}
	\HH_{\sqrt 2\delta }^{n}({\tilde K})&\le C\sum_{i\in\NN} (\sqrt 2r_i)^n(\epsilon/(\sqrt{2}r_i)+1)\le C\sum_{i\in\NN} \big (\epsilon r_i^{n-1}+ r_i^n\big)\le C(\epsilon+\delta)\sum_{i\in\NN} r_i^{n-1}\\&\le C(\epsilon+\delta)(\HH^{n-1}(K)+\epsilon).
\end{align*}
The conclusion follows letting $\epsilon\searrow 0$.
\end{proof}
\begin{lem}\label{ccsnasno}
It holds that 
$$
\nu_{\GG_f}^{{ u}}(x,t)=(\nu_f^{ u}(x),0),\quad\text{ for $|\DIFF\nchi_{\GG_f}|$-a.e.\ $(x,t)\in (D\cap J_f)\times\RR$}. 
$$
\end{lem}
\begin{proof}
Notice that, by \eqref{diffandper}, $(\pi_*^1|\DIFF\nchi_{\GG_f}|)\mres J_f=|\DIFF f|\mres J_f$. Then, by \Cref{coarea}, it suffices to show the claim for $x\in D\cap J_f\cap\FF\{f>s\}$, for some $s\in\RR$.
By \Cref{negl}, we can use a partitioning argument to reduce ourselves to prove the claim on $K\times I$, where $K\subseteq D\cap J_f\cap\FF\{f>s\}$ is compact with $\HH^{n-1}(K)<\infty$ and $I=(a,b)\subseteq\RR$ is an open interval such that for every $x\in K$, $\bar I\subseteq (f^\wedge(x),f^\vee(x))$ and $s\in I$. We can also assume, by \cite[Lemma 3.27]{BGBV}, that $\nu_f^u=\nu_{\{f>s\}}^u$ on $K$.
 
By a suitable modification of \Cref{remarkfrequenteAM}, we see that for $|\DIFF\nchi_{\GG_f}|$-a.e.\ $(x,t)$ with $x\in D$, it holds that $v\defeq (A(x)u,\pi^2)$ is a system of good coordinates for $\GG_f$ at $(x,t)$. Also,  for $|\DIFF\nchi_{\GG_f}|$-a.e.\ $(x,t)$ the conclusion of \cite[Proposition 4.8]{bru2021constancy} holds at $(x,t)$.  Also, for $|\DIFF\nchi_{\{f>s\}}|$-a.e.\ $x$ the analogous conclusions for $\{f>s\}$ are in place at $x$.
Take a point  $p=(x,t)\in\FF\GG_f$ of density $1$ for $K\times I$ with respect to $|\DIFF\nchi_{\GG_f}|$ such that $x$ is of density $1$ for $K$ with respect to $|\DIFF\nchi_{\{f>s\}}|$ and satisfying the conclusions above. 
Then for some sequence $\{r_i\}_i$, $r_i\searrow 0$, $$\big(\XX\times\RR,r_i^{-1}(\dist\times\dist_e),(\mass\otimes\mathcal L^1)_p^{r_i},p,\GG_f\big)\to(\RR^{n+1},\dist_e,\underline{\mathcal L}^{n+1},0,H),$$
in a realization $(\ZZ\times\RR,\dist_{\ZZ\times\RR})$, where $$H\coloneqq\big\{y\in\RR^{n+1}\;\big|\;y\cdot\nu(x,t)\geq 0\big\},$$
where $\nu(x,t)$ is given by \eqref{eq:def_nu_in_good_coord} for $v$. 
Also, we have that 
$$\big(\XX,r_i^{-1}\dist,\mass_x^{r_i},x,\FF\{f>s\}\big)\to(\RR^{n},\dist_e,\underline{\mathcal L}^{n},0,H'),$$where $H'$ is given by $$H'\coloneqq\big\{y\in\RR^{n}\;\big|\;y\cdot\mu(x)\geq 0\big\},$$
where $\mu(x)$ is given by \eqref{eq:def_nu_in_good_coord} for $A(x)u$. 
Up to taking a non-relabelled subsequence, we assume that  \(|\DIFF\nchi_{\GG_f}|\) weakly converges to \(\frac{1}{\omega_{n+1}}\HH^n\mres\partial H\) in duality with \(C_{\rm bs}(\ZZ\times\RR)\) and  \(|\DIFF\nchi_{\{f>s\}}|\) weakly converges to \(\frac{1}{\omega_{n}}\HH^{n-1}\mres\partial H'\) in duality with \(C_{\rm bs}(\ZZ)\). 
Now, as we will prove in the paragraph below, we have that $(\HH^{n-1}\mres K)\otimes \HH^1=\HH^n\mres(K\times \RR)$.
A simple argument of weak convergence, based on \Cref{thm:repr_per} and exploiting the density assumption, yields that $H=H'\times \RR$, whence the result follows. Indeed, we would have $\nu_{n+1}(x,t)=(\nu_{\mathcal{G}_f}^v)_{n+1}(x,t)=(\nu_{\mathcal{G}_f}^u)_{n+1}(x,t)=0$.

Now we prove the coarea formula claimed in the paragraph above exploiting the rectifiability result of \cite{bru2019rectifiability} and the fact that $K\subseteq \mathcal{F}\{f>s\}$. Take $E$ a set of finite perimeter.
Fix $\epsilon>0$.
We use \cite[Theorem 4.1]{bru2019rectifiability} (see also \cite[Remark 4.3]{bru2019rectifiability}) to write, up to $\HH^{n-1}$-negligible subsets, $\FF E=\bigcup_k E_k$, where $E_k$ are pairwise disjoint Borel subsets of $\FF E$ such that, for every $k$, $E_k$ is $(1+\epsilon)$-bi-Lipschitz to some Borel subset of $\RR^{n-1}$. Say that, for every $k$, there exists $g_k: E_k\rightarrow \RR^{n-1}$ $(1+\epsilon)$-bi-Lipschitz with its image. Call also $f_k: E_k\times\RR\rightarrow \RR^{n-1}\times \RR$ the map $(x,t)\mapsto (g_k(x),t)$.
We have that, if $\psi:\XX\times\RR\rightarrow[0,1]$ is Borel, then
\begin{align*}
\int_{\FF E\times\RR}\psi(x,t)\,\dd\HH^n(x,t)&=\sum_{k\in\NN} \int_{E_k\times \RR}\psi(x,t)\,\dd\HH^n(x,t)
\\&=\sum_{k\in\NN} \int_{g_k( E_k)\times \RR } \psi(g_k^{-1}(y),t)\,\dd((f_k)_*\HH^n)(y,t),
\end{align*}
where we used that if $N\subseteq\FF E$ is $\HH^{n-1}$-negligible, then $N\times\RR$ is $\HH^n$-negligible, thanks to a simple covering argument.
Now notice that, as $g_k$ is $(1+\epsilon)$-bi-Lipschitz, we have that, on their natural domains,
$$
\left(\frac{1}{1+\epsilon}\right)^{n-1}\HH^{n-1}\le (g_k)_*\HH^{n-1}\le \left(\frac{1}{1-\epsilon}\right)^{n-1}\HH^{n-1}$$ and $$ \left(\frac{1}{1+\epsilon}\right)^{n}\HH^n\le (f_k)_*\HH^n\le \left(\frac{1}{1-\epsilon}\right)^{n}\HH^n.
$$
Therefore, using Fubini's theorem in $\RR^n$, setting $\tilde\psi(y,t)\defeq\psi(g_k^{-1}(y),t)$ and denoting $C_\epsilon$ a constant, that may vary from line to line but that depends only on $\epsilon$ and $n$ and such that $C_\epsilon\rightarrow 1$ as $\epsilon\searrow 0$,
\begin{align*}
   \int_{g_k(E_k)\times \RR } \tilde\psi(y,t)\,\dd((f_k)_*\HH^n)(y,t)&\le  C_\epsilon \int_{g_k( E_k)\times \RR } \tilde\psi(y,t)\,\dd\HH^n(y,t)   
  \\&= C_\epsilon \int_{g_k(E_k) }\int_\RR \tilde\psi(y,t)\,\dd\HH^1(t)\,\dd\HH^{n-1}(y)
  \\&\le C_\epsilon \int_{g_k( E_k) }\int_\RR \tilde\psi(y,t)\,\dd\HH^1(t)\,\dd{((g_k)_*\HH^{n-1})}(y)
  \\&= C_\epsilon\int_{ E_k }\int_\RR \psi(x,t)\,\dd\HH^1(t)\,\dd\HH^{n-1}(x).
\end{align*}
All in all,
\begin{align*}
\int_{ \FF E\times\RR}\psi(x,t)\,\dd\HH^n(x,t)&\le \sum_k
C_\epsilon\int_{ E_k }\int_\RR \psi(x,t)\,\dd\HH^1(t)\,\dd\HH^{n-1}(x)
\\&=C_\epsilon \int_{ \FF E }\int_\RR \psi(x,t)\,\dd\HH^1(t)\,\dd\HH^{n-1}(x).
\end{align*}
As $\epsilon>0$ was arbitrary, we obtain that
$$
\int_{\FF E\times \RR}\psi(x,t)\,\dd\HH^n(x,t)\le \int_{\FF E}\int_\RR \psi(x,t)\,\dd {\HH^1(t)}\,\dd\HH^{n-1}(x).
$$
The opposite inequality is obtained similarly.
\end{proof}

\subsection{Proof of the main results}
We are now ready to prove the main theorems of this note.

\begin{proof}[Proof of \Cref{thm:Principale1}]
If $f\in \BV_{\rm loc}(\XX)$, then $\GG_f$ has locally finite perimeter thanks to the proof of item (a) in \cite[Theorem 5.1]{Ambrosio-Pinamonti-Speight15}. Conversely, assume that $\GG_f$ has locally finite perimeter. Then, the argument in the proof of item (b) of \cite[Theorem 5.1]{Ambrosio-Pinamonti-Speight15} yields that for any $x\in\XX$ and $r>0$,
$$
\int_\RR |\DIFF\nchi_{\{f>t\}}|(B_{2r}(x))\,\dd t<\infty.
$$
Now we take $t_0\in (0,\infty)$ big enough so that $\mass(\{f>t_0\}\cap B_r(x))\le \min\{1,\mass(\{f\le t_0\}\cap B_r(x))\}$  and
$\mass(\{f<-t_0\}\cap B_r(x))\le 
\min\{1,\mass(\{f\ge -t_0\}\cap B_r(x))\}$.  This is possible as $f\in L^0(\mass)$.
Thus, taking into account that for $\mathcal L^1$-a.e.\ $t$,
$ |\DIFF\nchi_{\{f>t\}}|= |\DIFF\nchi_{\{f<t\}}|$, we obtain from the relative isoperimetric inequality \eqref{isop} (that holds with $\lambda=1$ on finite-dimensional $\RCD$ spaces) that 
$$
\int_{t_0}^\infty \mass(\{f>t\}\cap B_r(x))\,\dd t<\infty\qquad\text{and}\qquad \int_{-\infty}^{-t_0} \mass(\{f<t\}\cap B_r(x))\,\dd t<\infty.
$$
This implies $f\in L^1_{\rm loc}(\XX)$ by Fubini's theorem. By \Cref{coarea}, it also follows that $f\in\BV_{\rm loc}(\XX)$.

The last conclusion is an immediate consequence of \Cref{area} and \Cref{prop:Reduction}, for what concerns the absolutely continuous part. For what concerns the equality on the jump part and the Cantor part, it directly follows from \eqref{diffandper}.
\end{proof}

\begin{proof}[Proof of \Cref{thmmain1}]
We first show that 
$$
\nu_{\GG_f}^{{ u}}(x,t)=\Bigg(\sqrt{\frac{1}{1+g_f^2}}g_f\nu_f^{ u},-\sqrt{\frac{1}{1+g_f^2}}\,\Bigg)(x),\quad\text{ for }|\DIFF\nchi_{\GG_f}|\text{-a.e.\ }(x,t)\in (D\setminus(J_f\cup C_f))\times\RR.
$$
Recall that \eqref{diffandper} and \Cref{prop:def_set_C_f} imply that we can reduce ourselves to show the claim for $|\DIFF\nchi_{\GG_f}|$-a.e.\ $(x,t)\in (D_f\setminus(C_f\cup J_f))\times\RR$. 

For $|\DIFF\nchi_{\GG_f}|$-a.e.\ $(x,t)\in \big(\{g_f=0\}\cap(D_f\setminus(C_f\cup J_f))\big)\times\RR$, by \Cref{area}, \Cref{cdskcasn} and \eqref{diffandper} it holds that $(\nu_{\GG_f}^u(x,\bar f(x)))_{n+1}=-1$. 
The claim is then proved at $|\DIFF\nchi_{\GG_f}|$-a.e.\ $(x,t)\in \big(\{g_f=0\}\cap(D_f\setminus C_f)\big)\times\RR$ by the following fact. By \Cref{prop:def_set_C_f}, for $x\in D_f$, at $(x,\bar{f}(x))$, $v\defeq (A(x)u,\pi^2)$ is a system of good coordinates for $\GG_f$, see also \Cref{remarkfrequenteAM}. Also, if $\nu(x,\bar f(x))$ is computed as in item v) of \Cref{prop:def_set_C_f}, it holds that $$\nu(p)= \big(A(x)(\nu_{\GG_f}^u(p))_{1,\dots,n}, (\nu_{\GG_f}^u(p))_{n+1}\big),\qquad\text{for $|\DIFF\nchi_{\GG_f}|$-a.e.\ $p=(x,t)\in D_f\times\RR$},$$
and that $|\nu(p)|=1$. Recall that $A(x)$ is invertible, whence the conclusion follows.

Now we show the claim at $|\DIFF\nchi_{\GG_f}|$-a.e.\ $(x,t)$ with $x\in\{g_f>0\}\cap\big(D_f\setminus(C_f\cup J_f)\big)$. Notice that on $\{g_f>0\}\cap\big(D_f\setminus(C_f\cup J_f)\big)$ it holds that $\mass\ll|\DIFF f|\ll\mass$. Therefore, by \Cref{csmkc}, taking into account \Cref{area}, \Cref{cdskcasn} and \eqref{diffandper}, we have the claim.

The fact that 
$$
\nu_{\GG_f}^{{ u}}(x,t)=\big(\nu_f^u(x),0\big),\quad\text{ for }|\DIFF\nchi_{\GG_f}|\text{-a.e.\ }(x,t)\in (D\cap(J_f\cup C_f))\times\RR
$$
is \Cref{ccsnasno} together with \cite[Theorem 3.7 and Lemma 3.8]{ABPrank}.
\end{proof}

\begin{proof}[Proof of \Cref{thmmain2}]
Items $i)$ and $ii)$ can be proved using \Cref{thmmain1}, \Cref{thm:Principale1}, and \Cref{lem:char_G_f}. Item $iv)$ follows from \Cref{ccsnasno}.

We show now item  $iii)$. 
By the representation formula, we write 
\begin{align*}
&\int_{(D\cap J_f)\times\RR}\varphi(x,t) \big(\nu_{\GG_f}^u(x,t)\big)_{i} \dd{|\DIFF\nchi_{\GG_f}|(x,t)}\\&\qquad=\int_{(D\cap J_f)\times\RR}\varphi(x,t) \big(\nu_{f}^u(x)\big)_{i} \nchi_{\partial^*\GG_f}(x,t)\Theta_n(\mass,x)\,\dd\HH^n(x,t),
\end{align*}
where we used that $\Theta_n(\mass,x)=\Theta_{n+1}\big(\mass\otimes\HH^1,(x,t)\big)$ and \Cref{ccsnasno}. Now notice that if $N\subseteq J_f$ is such that $\HH^{n-1}(N)=0$, then $\HH^n(N\times\RR)=0$. This can be proved with an easy covering argument. 

Therefore, taking into account also \Cref{lem:char_G_f} and \Cref{coarea}, we reduce ourselves to prove that for every $\psi:D\times \RR\rightarrow [0,1]$ Borel, we have that for $\HH^1$-a.e.\ $s\in\RR$
$$
   \int_{(D\cap \FF E_s\cap J_f)\times\RR}\psi(x,t)\,\dd\HH^n(x,t)=\int_{D\cap \FF E_s\cap J_f}\int_\RR\psi(x,t)\,\dd t\,\dd\HH^{n-1}(x),
$$
where $E_s\defeq\{f>s\}$. Fix $s$ such that $E_s$ has finite perimeter. The claim is equivalent to $\HH^{n}\mres (\FF E_s\times\RR)=(\HH^{n-1}\mres \FF E_s)\otimes \HH^1$, which has been proved at the end of the proof of \Cref{ccsnasno}.
\end{proof}

\end{document}